\documentclass[11pt]{article}
\usepackage{a4}
\usepackage{graphicx}
\usepackage{parskip}

\usepackage{amsfonts}
\usepackage{natbib}
\usepackage{algorithm}
\usepackage{algorithmic}
\usepackage{amsmath}
\usepackage{setspace}

\usepackage{latexsym}
\usepackage{amsthm}
\usepackage{rotating}
\usepackage[pdftex]{color}
\usepackage{csvsimple}
\usepackage{verbatim} 

\newtheorem{theorem}{Theorem}[section]

\newtheorem{lemma}{Lemma}[section]

\newtheorem{condition}{Condition}[section]

\theoremstyle{remark}
\newtheorem{remark}{Remark}[section]
\newtheorem{example}{Example}[section]

\newcommand{\R}{\mathbb{R}}

\newcommand{\argmin}{\mathop{\arg \min}\limits}

\newcommand{\PP} {{  \rm I\hskip-0.22em P}}

\newcommand{\EE} {{\rm I\hskip-0.48em E}}

\newcommand{\KX}{K_{X}}
\newcommand{\KG}{K_{g}}
\newcommand{\Jh}{J_{h}}
\newcommand{\CC}{C}
\newcommand{\kappaUno}{\kappa_{0}}
\newcommand{\kappaUnoG}{\kappa_{0}^{\prime}}
\newcommand{\kappaUnoH}{\kappa_{0}^{\prime\prime}}
\newcommand{\deltaUno}{\delta_{1}}
\newcommand{\deltaUnoP}{\delta_{1}^{\prime}}
\newcommand{\kappaUnoP}{\kappa_{1}^{\prime}}
\newcommand{\kappaTre}{\kappa_{2}}
\newcommand{\deltaUnoS}{\delta_{1}^{\prime\prime}}
\newcommand{\deltaDUE}{\delta_{2}}
\newcommand{\deltaTre}{\delta_{3}}

\newcommand{\deltaTOT}{\delta_{T}}
\newcommand{\kappaTOT}{\kappa_{T}}

\newcommand{\be}{\begin{eqnarray}}
\newcommand{\ee}{\end{eqnarray}}
\newcommand{\bes}{\begin{eqnarray*}}
\newcommand{\ees}{\end{eqnarray*}}

\usepackage{natbib}

\begin{document}
\onehalfspacing

\begin{center}
\LARGE{The partial linear model in high dimensions}\\\vspace{1.5cm}
 \normalsize{\textit{High-dimensional partial linear model}}
\vspace{1.5cm}

\Large{Patric M\"uller and Sara van de Geer \\Seminar f\"ur Statistik, ETH Zurich}
\end{center}

\newpage

\begin{abstract} Partial linear models have been widely used as flexible method for modelling
linear components in conjunction with non-parametric ones. Despite the presence of the non-parametric part,
the linear, parametric part can under certain conditions be estimated with parametric rate.
In this paper, we consider a high-dimensional linear part. We show that it can be estimated
with oracle rates, using the LASSO penalty for the linear part and a smoothness penalty for
the nonparametric part.
\end{abstract}


\section{Introduction}
Consider the partial linear model where the expectation of the response variable $Y$ depends on two predictors
$(X,Z)$. The dependence of $Y$ on the first explanatory variable $X$ is linear, whether a non-linear, unknown,
nuisance function defines the dependence on $Z$.\\
In order to deal with this model we need some special techniques which take into account its particular properties.
Simply neglecting the non-linear part and using the standard least square estimator does not lead to
the correct answers, on the other hand directly applying some non-linear regression (i.e.\ treating the linear
part as non-linear) is often too rough.\\
The partial linear model was studied extensively, see e.g.\ \cite {green1985semi}, \cite{wahba1990spline}
and more recently
\cite{hardle2007partially} and the references therein. 
 In \cite{Mammen97} the penalised least squares estimator is shown to
be consistent. Asymptotic normality for the linear part was also shown.
These results are valid in the low dimensional case, where the number of observations $n$ is much larger than the number of variables $p$ in the linear part i.e. $n\gg p$.\\
The high-dimensional case, where $n\ll p$, is more
challenging. The model is now underspecified and the methods proposed in \cite{Mammen97} can not be
directly applied. We need some more techniques in order to overcome these new difficulties.\\
There is a large body of work on linear high dimensional models. A common approach is to construct penalised
estimators like the LASSO (least absolute shrinkage and selection operator proposed in \cite{Tibshirani96}) or
the ridge regression and elastic net (see \cite{zou05}).
The LASSO is widely studied (see e.g.\ \cite{koltchinskii2011oracle} and  \cite{BvdG2011} and references therein) and gives remarkable results
in sparse contexts.\\
In this paper we study the high-dimensional partial linear model.
Our main contribution is to combine the methods used in \cite{Mammen97} for the low dimensional case with the
standard LASSO and obtain an estimator that can deal with and overcome both extra difficulties given from the
high-dimensional and the non-linear part of the model.\\
We prove theoretical results giving bounds for both the prediction and the estimation error.
In the last part of the paper we present simulation results. We compare the performance of our method with the
standard LASSO with \break (un-)known nuisance function.\\
\medskip \\
The paper is organised as follows.
We begin in Section \ref{Notation.section} with notation and model description. The required conditions are 
here listed and discussed. Section \ref{Results.section} 
describes our main results.  In Section \ref{NumericalRes.section} we present simulations. These confirm the 
previous theoretical results. Section \ref{Proffs.section} contains technical tools from empirical process 
theory and concentration inequalities. These tools are required for the proof of the results from Section 
\ref{Results.section}, which are also part of Section \ref{Proffs.section}. A set of tables 
with detailed results of the simulations can be found in Section \ref{NumRes.Appendix}

\section{Notation and model assumption}\label{Notation.section}
In this section we describe more in detail the model we study. After the formal description of the model, we
define some extra conditions we need in order to prove our results. For each condition we further add a short
comment on their strength. The notation we use is explained in the following subsection.

\subsection{Notation}
Through all the paper we always use $i\in\{1,\ldots,n\}$ and $j\in\{1,\ldots,p\}$.

For a vector $\beta\in\R^p$ the $\ell_1$-norm is $\|\beta\|_1:=\sum_{j=1}^p |\beta_j|$.
For a fixed $\beta^0 \in \R^p$, we denote with $S_0$ the set of all non-zero components of $\beta^0\in\R^p$ and 
with $s_0$ its cardinality. Define for all $\beta \in \R^p$, the vector $\beta_{S_0}\in\R^p$ where 
$(\beta_{S_0})_i=|\textrm{sign}(\beta^0_i)|\beta_i$, i.e. the 
components in $S_0^c$ are set to $0$. Define $\beta_{S_0^c} := \beta - \beta_{S_0} $.

Let $x\in\R^p$ and $z\in\R^d$ be high-dimensional and low-dimensional random variables, respectively.
Denote by $Q$ the distribution of $(x,z)$ and by $Q_x(Q_z)$ the marginal distribution of $x(z)$.
We let $\{ (x_i, z_i)\}_{i=1}^n $ be $n$ i.i.d.\ copies of $(x,z)$.
In matrix notation we denote with $X\in
M_{n\times p}(\R)$ the matrix with rows $x_i\in\R^p$ ($i=1 , \ldots , n$) and columns $X_j\in\R^n$
($j=1 , \ldots , p $).  With $X_{ij}$ we denote the $(i,j)^{th}$ component of $X$.
Similarly, we let $Z$ be the $n \times d $-matrix with rows $z_i $ ($i=1 , \ldots , n $).

For $f: \R^{p+d} \rightarrow \R$ a measurable function, we define, with some abuse of notation
$f(X,Z):=(f_1,\ldots,f_n)^T:=\left(f(x_1, z_1 ),\ldots,f(x_n, z_n ) \right)^T$. The squared $L_2$-norm is
$\|f\|^2:=\EE f^2(x,z)$, where $\EE$ is the expectation with respect to the distribution $Q$. The sup-norm 
is $\|f\|_\infty:=\sup_{\rm x , z} |f({\rm x, z } )|$ and the squared empirical norm is defined as 
$$
\|f\|_n^2:=\frac{1}{n} \sum_{i=1}^n f^2(x_i, z_i )~.
$$
More generally, for a vector $v \in \R^n$, we write $\| v \|_n^2 := v^T v / n $. 

\subsection{Model description and motivation}
Let ${\cal G}$ be a linear subspace of $L_2(Q_z)$, $g^0\in\mathcal{G}$ and $\beta^0\in\R^p$ be fixed.\\
The response variable $y\in\R$ depends linearly on $x$ and in a non-parametric way on $z$. 
The observed variables are i.i.d.\ copies of $(x,z,y)$, whereas $(\beta^0,g^0)$ are unknown parameters. The resulting model is then
$$ y=x\beta^0+g^0(z)+e~, $$
where $e$ is the error term which can be interpreted as noise. 

The observations are denoted by $\{ (x_i , z_i , y_i) \}_{i=1}^n $ ($n$ i.i.d.\ copies  of $(x,z,y)$). In matrix 
notation we have
\be \label{Model.equation}
Y=X\beta^0+g^0(Z)+E~\in\R^n~,
\ee
where $g^0(Z):=(g^0(z_1),\ldots,g^0(z_n))^T\in\R^n$, $E:=
(e_1 , \ldots , e_n)^T \in \R^n$ and 
 $Y:=(y_1,\ldots,y_n)^T\in\R^n$.
\begin{remark}
In this paper we take $p\gg n$. Consequently we have a high-dimensional linear model with nuisance
function.
\end{remark}
Semi-parametric models are quite common in low dimensional contexts. In high-dimensional settings linear models 
(possibly after suitable coordinate transformation) are widely analysed.
In some cases however there is strong evidence to believe that one or some variables have a non-linear influence
on the response. Consider the following (fictive) example:
\begin{example}
Take $Y$ is the yield of some genetically modified plants; $X$ represents the gene expression data of
plants; $Z$ are the factors 'water' and 'temperature'. It is
reasonable to assume that $Y$ linearly depends on $X$, but not on $Z$ (too
much or few water and too high or low temperature reduce the yield). Model \eqref{Model.equation} provides a 
reasonable approach to describe this problem.
\end{example} 

\subsection{Estimator and model assumptions}

Let $J(\cdot,\cdot)$ be a (semi-)scalar product on $\mathcal{G}$ and $J(\cdot )$ be the corresponding
\break (semi-)norm.
We define the doubly penalised least square estimator $(\hat\beta,\hat g)$ as:
\be \label{Estimator.Definition.Equation}
(\hat\beta,\hat g):=\argmin_{(\beta,g)\in\R^p\times\mathcal{G}}
\left\{||Y-X\beta-g||_n^2+\lambda||\beta||_1+\mu^2J^2(g)\right\}
\ee
We thus propose a doubly penalised least square estimator. The idea behind it is that the $\ell_1$-penalisation on the 
linear coefficients controls the sparsity of $\hat \beta$, whether the second penalty term keeps the 
estimation of $g$ under control.
In our hope our estimator has the desired properties given from each of the two different penalisations
and this, possibly, without paying a too high price in terms of prediction and estimation error.

\begin{remark}
If $g^0$ is known, the restricted version of  \eqref{Estimator.Definition.Equation} with
$g=g^0$ gives the standard LASSO for 
the linear model (see \eqref{Def_Standard_Lasso.equation}). 
\end{remark}

\begin{remark} \label{def_J.remark}
When $z \in \R$ (i.e., $d=1$) a common choice for $J^2(g)$ is the $L_2$-norm of the second derivative, $J^2(g)=\int_\R (g''(z))^2 d\nu(z)$,
where $\nu$ is some probability measure (see e.g. \cite{wahba1990spline}, 
\cite{green1994nonparametric}). 
\end{remark}

In the following lines we give a set of conditions we assume in this paper. Their strength is also shortly
discussed.

\vskip .1in
\begin{condition}[Gaussian condition]\label{Gaussian.condition} ~\\
The errors $e_1 , \ldots , e_n$ are i.i.d. standard Gaussian random variables independent of $\{ x_i, z_i
\}_{i=1}^n $.
\end{condition}
Gaussian errors is a quite strong, but not unusual assumption. In any case this condition can be easily relaxed
to sub-Gaussian errors. We have for simplicity assumed unit variance for the errors. In practice, the variance
of the errors will generally not be know. The tuning parameters then will be scaled by
an estimated error standard deviation/variance. 

\vskip .1in
\begin{condition}[Design condition]\label{design.condition} ~\\
For some constant $\KX$ it holds that
$$ \max_{i,j} |X_{ij} | \le \KX . $$
\end{condition}
A bound on the $X$-values is a quite restrictive assumption. 
However we can often approximate a
non-bounded distribution with its truncated version.

\vskip .1in
Define
$$
h(Z) := \EE [X|Z] , \ \tilde X:= X-\EE[X|Z]~.
$$
\begin{condition}[Eigenvalue condition]\label{eigenvalue.condition} ~\\
The smallest eigenvalue $\Lambda_{\tilde X, {\rm min}}^2 $ of $\EE \tilde X^T \tilde X / n$ is positive.\\
Furthermore the largest eigenvalue  of $\EE h^T h / n $, denoted by $\Lambda_{h, {\rm max}}^2$, is assumed
to be finite.
\end{condition}
This condition ensures that there is enough information in the data to identify the parameters in the linear part.

\vskip .1in
Let for each $u>0$ $N(u,\mathcal{G},\|\cdot\|_\infty)$ be the smallest value of $N$ such that there exists
$G$, a subset of $\mathcal{G}$ of cardinality $N$, for which 
$$
\sup_{g\in\mathcal{G}} \min_{\tilde g\in G} \|g-\tilde g\|_\infty \leq u~.
$$
Then $H(\cdot ,\mathcal{G},\|\cdot\|_\infty):=\log N(\cdot ,\mathcal{G},\|\cdot\|_\infty)$ is called the entropy of
$\mathcal{G}$ for the supremum norm.
\begin{condition}[Entropy condition] \label{entropy.condition} ~\\
For some constants $A$ and $m > 1/2$ one has
$${\cal H} \bigl(u, \{ g:\  \| g \| \le 1 , \ J(g) \le 1 \} , \| \cdot \|_{\infty} \bigr)\le A u^{-1/m} , \ u
> 0 . $$
\end{condition}
This condition summarizes the assumed ``smoothness" of the nonparametric part. 
For example, when $Q_z$ has support in  $[0,1]$, the choice of Remark \ref{def_J.remark} with $\nu$ Lebesgue measure on $[0,1]$ has $m=2$.

\vskip .1in
\begin{condition}[Penalty condition]\label{penalty.condition} ~\\
For some constant $\KG$, it holds that
$$\sup_{\| g \| \le 1 , \ J(g) \le 1 } \| g \|_{\infty} \le \KG. $$
\end{condition}
This condition states that, if the $L_2$ norm of $g$ and $J(g)$ are bounded, then the supremum norm of $g$
is bounded as well. This avoids functions with high and very steep peaks.

\vskip .1in
\begin{condition}[$J$ and $h$ condition]\label{Jh.condition}~\\
For some constant $\Jh$, it holds that
$$J(h) \le \Jh. $$
\end{condition}
This condition is fulfilled for $\Jh=0$ if $x$ and $z$ are independent and $J$ is chosen as in Remark 
\ref{def_J.remark}


\section{Results}\label{Results.section}
In this section we present our main results. We provide theoretical guarantees in term of prediction and 
estimation for estimator \eqref{Estimator.Definition.Equation}. Our first theorem proves the convergence of our 
method, whether the second leads to an oracle result. In particular Theorem \ref{2.theorem} shows that, up to a 
constants, our method performs as good as if the nuisance function $g^0$ were known.
\bigskip \\
We use the short-hand notation 
$$
{\cal F}:=\{f=X\beta+g: \ g\in {\cal G}, \ \beta\in \R ^p \},
$$
and
$$
f^0:=X\beta^0+g^0 .
$$

\begin{remark} \label{OrthogonalPorj.remark} (Orthogonal decomposition)\\
For every $f,f^0\in \mathcal{F}$, $f-f^0=\tilde X(\beta-\beta^0)+\big[h(\beta-\beta^0)+g-g^0\big]$ is an
orthogonal decomposition i.e.:
\be
\|f-f^0\|^2&=&||\tilde X (\beta-\beta^0)\|^2+\|h(\beta-\beta^0)+g-g^0\|^2
\label{f_Orthogonal_decomposition.equation}
\ee
\end{remark}

Define now
$$
\tau_{\mu,R}(f):=\tau(f) := \lambda \| \beta \|_1 /( R \sqrt {\delta_0 / 2} )  + \sqrt { \| X \beta   + g \|^2 +
\mu^2 J^2 (g)},
$$
where $\delta_0$ is a fixed small constant. Values and optimization of the constants are in this paper of minor
interest. However we give hereafter indications and or values for the mentioned constants.
Note that $\tau$ is a (semi-)norm. This is however not used for our results, we will only use
in Lemma \ref{Teo_T.lemma}
the cone property $\tau (s f)= s \tau (f)$ for all $f $ and all $s > 0$.

\begin{theorem} \label{1.theorem} Assume Conditions \ref{Gaussian.condition}, \ref{design.condition}, 
\ref{eigenvalue.condition}, \ref{entropy.condition} and \ref{penalty.condition}.
Let $\deltaTOT$ be a (small) constant, and $\kappaTOT$ be a (large) constant,
suitably chosen, depending on $\Lambda_{\tilde X, {\rm min}}$, $\Lambda_{h, {\rm max}}$,
$\KX$, $\Jh$, $A$ and $m$ only. Take 
\be \label{condition_mu_Thm.equation}
 \kappaTOT n^{-\frac{2m}{2m+1} } \le  \mu^2 \le \frac{ \delta_0 R^2}{2( 1+ J(g^0)+\Jh)^2 }.
\ee
Assume $R^2 \le \lambda \le 1$, 
\be \label{Assumption_R2Lambda2}
{4 \lambda^2 s_0 \over \Lambda_{\tilde X, {\rm min}}^2 } \le \delta_0 R^2~,
\ee
\be \label{Assumption_lambda0.equation}
\lambda \ge \frac{ \sqrt { 2 \log (2p) / n}  }{ \deltaTOT }
\ee
 and
\be \label{Assumption_EV_Lambda.equation}
\Lambda := 1+ {\Lambda_{h, {\rm max}} \over  \Lambda_{\tilde X, {\rm min}} }\leq  \mu^{-1} ~.
\ee
Then with probability at least $1- 3 \exp[-n \deltaTOT^2 \mu^2 ] $, it holds that $\tau( \hat f - f^0 ) \le R $.
\end{theorem}

\textbf{Asymptotics} To obtain a clearer picture of the result, let us rephrase it in an asymptotic framework,
where $1/ \Lambda_{\tilde X, {\rm min}}$, $\Lambda_{h, {\rm max}}$,
$\KX$, $\Jh$, $J(g^0)$ and $A$ are all bounded by a fixed constant and  $m$ is fixed. Then one sees that
for $\lambda$ and $\mu$ having the usual order $\lambda \asymp \sqrt {\log p / n }$, respectively 
$\mu \asymp n^{-{m \over 2m+1}}$, and when
$s_0 \log p /n = o (n^{-{2m \over 2m+1}})$ (that is, the oracle rate for the linear part established in 
Theorem \ref{2.theorem} is faster than the rate of convergence for estimating the nuisance part)
then the overall rate of convergence is $\tau(\hat f - f^0) ={\mathcal O}_{\PP}( n^{-{m \over 2m+1}} )$. 
This in particular implies $\| \hat g - g^0 \|= {\mathcal O}_{\PP}( n^{-{m \over 2m+1}} )$ and
$J (\hat g) = {\mathcal O}_{\PP} (1)$. 

\textbf{Discussion on the constants}
The results presented in Theorem \ref{1.theorem} are valid for $\delta_0=1 /253$.
$\mu$ and $R$ are intended to be small constants. The constants $\kappaTOT$ and $\deltaTOT$ are big and small 
respectively. They can be for example defined along the lines of \eqref{kappaTOT.equation} and 
\eqref{deltaTOT.equation}. We remind that the optimisation of the constants in this paper is of minor interest.

\begin{theorem} \label{2.theorem}
Assume the same conditions of Theorem \ref{1.theorem} with the constants
$\delta_0$ and $\deltaTOT$ small enough. Then with probability
at least $1-(2/p)+6 \exp[-n \deltaTOT^2 \mu^2]$
$$
\|\tilde X (\hat\beta-\beta^0)\|_n^2 +\lambda \|\hat\beta-\beta^0\|_1/4 \leq
\frac{s_0\lambda^2}{\Lambda_{\tilde X,{\rm min}}^2 }
$$
\end{theorem}

\begin{remark}
This theorem contains two important results, we get prediction results
$$
\|\tilde X (\hat\beta-\beta^0)\|_n^2 \leq \frac{s_0\lambda^2}{\Lambda_{\tilde X, {\rm min}}^2}
$$
and estimation results
$$
\|\hat\beta-\beta^0\|_1 \leq \frac{4 s_0\lambda}{\Lambda_{\tilde X, {\rm min}}^2}.
$$
\end{remark}

\begin{remark}
Theorem \ref{2.theorem} says that 
 one can estimate
$\beta^0$ in $\|\cdot\|_1$-norm with the same rate as in the case where the nuisance parameter
$g^0$ is known.  In asymptotic terms, with $\lambda \asymp \sqrt {\log p / n} $ and
$\mu \asymp n^{-{m \over 2m+1}}$, this rate is ${\mathcal O}_{\PP} ( s_0 \sqrt {\log p / n} )$.
One may verify moreover that the theoretical (out of sample) prediction error, $ \| X ( \hat \beta - \beta^0) \|^2$,
the empirical prediction (in sample) error
$ \| X ( \hat \beta - \beta^0 ) \|_n^2 $ and the squared $\ell_2$-error $\| \hat \beta - \beta^0 \|_2^2$ are all
of order $s_0 \log p / n $ in probability. 
\end{remark}

\section{Numerical results}\label{NumericalRes.section}
In this section we present the results of a pseudo real data study. We compare the performance of our estimator 
\eqref{Estimator.Definition.Equation} with the LASSO estimator.
\medskip\\
The following model
$$
\tilde Y=X\beta^0+E
$$
is a simpler version of \eqref{Model.equation}, where $g^0$ is known and is the zero function. A widely used
estimator  in this  case is the Lasso (least absolute shrinkage and selection operator) $\hat \beta^{lasso}$, 
(\cite{Tibshirani96}). That is:
$$
\hat\beta^{lasso}:=\argmin_{\beta} \|Y-X\beta\|_n^2 +\lambda \|\beta\|_1
$$
Under some compatibility assumptions (see \cite{vandeG07}, \cite{koltch09b}, \cite{koltch09a} and \cite{bickel2009sal}) we have,
with high probability, the following performance for
the Lasso:
$$
\|X(\hat\beta^{lasso}-\beta^0)\|_n^2+\lambda\|\hat\beta^{lasso}-\beta^0\|_1\leq \frac{4 \lambda^2
s_0}{\phi_0^2}~,
$$
where $\phi_0$ is the so called compatibility constant (see Theorem 6.1 and Corollary 6.2 in \cite{BvdG2011}).
This result is very similar to Theorem \ref{2.theorem} which imply
that our method should work, asymptotically, as good as the Lasso in the case where the function $g^0$ is 
known.\\
In this section we will therefore compare our estimator with the Lasso.

\subsection{Dataset and settings}
We construct the pseudo real dataset for Model \eqref{Model.equation} as follows:\\
We take $X$ as a $n \times p$ matrix, from real data. $X$ is obtained by randomly picking out $p$ 
components from one of the following datasets.
\begin{itemize}
 \item Leukemia, also used in \cite{dettling04} and \cite{BuMando}.
 \item Prostate, also used in \cite{dettling04} and \cite{BuMando}.
\end{itemize}
$Z$ is simulated (see Remark \ref{Zsim.rmk}).
We furthermore analyse two different cases. We distinguish the case where $X$ and $Z$ are independent or 
dependent. To create dependency in the active set of $X$ we redefine the values of the first three 
columns of the matrix $X$ as follows:
$$
X_1:=2Z+V_1\quad,\quad X_2:=2Z^2+V_2\quad X_3:=-Z+V_3 ,
$$
where the $V_j$ are $n$-dimensional vectors with i.i.d normally distributed components. The resulting empirical 
correlation between $X_1$ and $Z$ is on average 0.74.

We let the active components of $\beta^0$ assume values  $\pm 1$ with probability $0.5$ each. Without loss of 
generality we take $S_0:=\{1,\ldots,s_0\}$, i.e. the first $s_0$ components of $\beta^0$ are different from 
zero.
\\
We let $z$ be of dimension $d:=1$. Consequently $Z\in \R^n$. 
\begin{remark}\label{Zsim.rmk}
In general it is (almost) impossible to properly estimate a real function from observational data in a region 
where there are no or very few observations. In order to avoid big gaps in $Z$ (intervals with 
few or none observations), we keep $z_i$ i.i.d. copies of $z \sim Unif[-0.5,0.5]$. Observe that actually
gaps do not cause very big problems in term of prediction. Only the estimation of $g$ could be imprecise in 
intervals with very few observations.\\
In any case, an appropriate prior standard transformation can be applied in order to make our data ``look 
better''.
\end{remark}

The semi-real data are then generated as 
\be \label{LinearModelBasic.equation}
Y = X \beta^0 +g^0(Z) + \varepsilon
\ee
We compare the following three estimators:
\begin{itemize}
 \item Lasso with known function $g^0$ (LK):
\be \label{LassoNoPrior.equation}
\hat\beta^{lasso~g^0}:=\argmin_{\beta} \|Y-g^0(z)-X\beta\|_n^2 +\lambda \|\beta\|_1
\ee
One can quickly notice that this is nothing else than Lasso for the linear high-dimensional model.
 \item Standard Lasso (LN):
\be \label{Def_Standard_Lasso.equation}
\hat\beta^{lasso}:=\argmin_{\beta} \|Y-X\beta\|_n^2 +\lambda \|\beta\|_1
\ee
\item Our estimator, the Doubly Penalised Lasso (DP):
\be \label{OurEstChapterResults.equation}
(\hat\beta,\hat g):=\argmin_{(\beta,g)\in\R^p\times\mathcal{G}}
\left\{||Y-X\beta-g||_n^2+\lambda||\beta||_1+\mu^2J^2(g)\right\}
\ee
\end{itemize}
For each one of the datasets and each estimator we fit 36 designs with 1000 repetitions each.
The designs are obtained by varying the following parameters:
\begin{itemize}
 \item [-] The dimension $p$, or number of variables in the model. In each simulation run, $p$ covariables are
chosen at random among all covariables in the dataset. We take $p$ either $250$ or $1000$.
 \item [-] The sparsity $s_0$. This denotes the number of non-zero components of $\beta_0$. We let $s_0$ be $5$
or $15$.
\item [-] The function  $g^0(z)$. We use the following three nuisance functions (see Figure \ref{g0.Figure}).
$$
g^0_1(z):= 0
$$
$$
g^0_2(z):= -20z^2-10
$$
$$
g^0_3(z):=3(e^{2z}+\sin(12z))
$$
Our intention is to have a `representative' sample among the bounded functions. Therefore we
choose an ''easy'' quadratic function ($g_2$) and a more complicated one ($g_3$). 
In order to give importance to both the linear and the non-linear part of the design the
functions $g_2$ and $g_3$ have comparable range of values as $X\beta^0$.\\
The trivial case where there is no nuisance function is represented by $g^0_1$.
 \item [-] The linear signal to noise ratio (lSNR), defined as 
$$
\text{lSNR} := \sqrt{\frac{\|X\beta^0\|_n^2}{\sigma^2}}~.
$$
The lSNR can be 2, 8 or 32.
\end{itemize}

\begin{remark}
We furthermore define the total signal to noise ratio:
$$
\text{tSNR} := \sqrt{\frac{\|X\beta^0+g^0\|_n^2}{\sigma^2}}.
$$
In our computations we fix lSNR and look at the corresponding tSNR value (and not the opposite). In such a way
the error term is independent of the magnitude of $g^0$.
Furthermore adding a nuisance function $g^0$ to model \eqref{LinearModelBasic.equation} increases at the
same time the $tSNR$ and the difficulty of the estimation of the parameters. This is somehow not ''fair''. 
The classical signal to noise ratio for the standard Lasso (in our case the Lasso with known $g^0$) is 
lSNR. Fixing lSNR allows us to compare our results with other papers, like \cite{BuMando}.
\end{remark}

\begin{center}
\begin{figure}
 \centering
 \includegraphics[width=180pt]{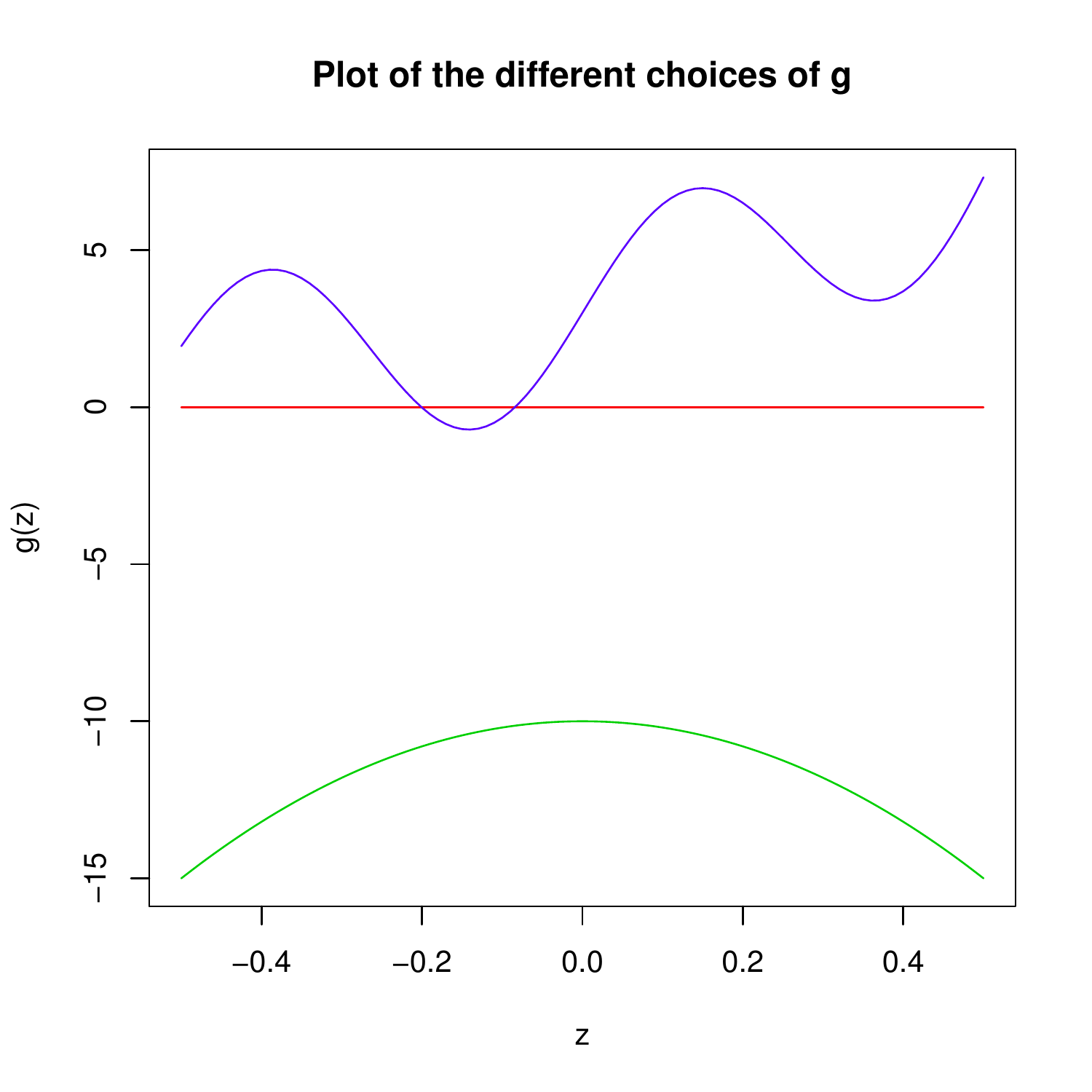}
\caption{The three functions used for testing the methods: $g_1(z)\equiv 0$, $g_2(z)=-20z^2-10$ and
$g_3(z)=3(e^{2z}+\sin(12z))$}
\label{g0.Figure}
\end{figure} 
\end{center}

Table \ref{scenario.Table} resumes the initial conditions.
\begin{table}[!htb]
\centering
\begin{tabular}[h]{|l|c|}
\hline
Setting parameter &   \\
\hline \hline
number $p$ of variables & 250 or 1000\\
\hline
sparsity $s_0$ & 5 or 20\\
\hline
linear signal to noise ratio (lSNR) & 2, 8 or 32 \\
\hline
function $g^0$ &$g^0_1$, $g^0_2$ or $g^0_3$\\
\hline
\end{tabular}
\caption{The parameter settings.}
\label{scenario.Table}
\end{table}

\textbf{Other settings:} 
A common choice for $J$ is $\int | g(z)''|^2 dz$. For computational reasons 
we take $J^2(g):=\int |g(z)''|^2 dz+c\int g^2 dz$, where $c$ is small
($\approx10^{-3}$). 
Note that $J$ is a norm. We choose $\mu:=n^{-2/5}/100$.

\subsection{Results}
Our aim is to compare our estimator with a well known estimator such the Lasso. In order to acquire a more 
global knowledge on the performance of our method different settings are chosen. In some of them the LK method 
performs very well (generally when $s_0$ is small and the lSNR is large) and in some others it is quite bad.
\\
In Tables \ref{ResultsLe1.table}-\ref{ResultsPr2.table} we summarize the results of the simulations.
The performances of LK (see \eqref{Def_Standard_Lasso.equation}), LN (see \eqref{LassoNoPrior.equation}) and 
DP \eqref{OurEstChapterResults.equation} are compared. The two different cases, where $X$ and $Z$ are 
independent and dependent are denoted with (DPi) and (DPd) respectively. Analogously we can define $LKi$ and 
$LKd$.\\
If $g^0$ is known, the dependence between $X$ and $Z$ does not play (given $X$) any role. Consequently LK scores
very similarly in the dependent and in the independent case. In order to make our table more readable 
we just take the average of the score of LKi and LKd and denote it with LK. The same consideration holds for LN.
\medskip\\
More in detail we compare the prediction error $\|\hat Y - Y_0\|_n$, the estimation error 
$\|\hat\beta-\beta^0\|_1$ and the true and false positive rate for $\hat\beta$ (TPR and FPR respectively). The 
error in the estimation of $g$, $\|\hat g-g^0\|_2$ is also given.
\medskip \\
Hereafter we summarize the findings of our simulations.
\subsubsection*{Prediction:}
As one can expect LK performs better than the other methods. This is not surprising because the nuisance model 
is more complex than the high-dimensional linear model. But this does not mean that our estimator is bad.
In fact it works only slightly worse than LK in prediction terms. Compared to LN, DPi and DPd works in any 
design (for non-zero nuisance function) remarkably better. Finally we can remark that DP provides 
\textit{slightly} better results if $X$ and $Z$ are independent, i.e. DPi has lower error than DPd. (See Figure 
\ref{Boxplot_Pred}).

\subsubsection*{Estimation:}
First of all note that DP is the only method estimating both $\beta^0$ and $g^0$.\\
As theoretically shown in Theorem \ref{2.theorem} LK and DPi and DPd provide similar estimation for 
$\beta^0$. Similarly as for the prediction LK performs slightly better than DPi and DPd. DPi and DPd are 
remarkably better than LN, when LK (and consequently DP) works well. When all four methods works bad, then LN 
is only slightly worse than our methods. Compare Figure \ref{Boxplot_Estim} and Tables 
\ref{ResultsLe1.table} and \ref{ResultsPr1.table} for quantitative results. The prediction of $\hat \beta$ is 
usually slightly worse in DPd than in DPi. This is not in contrast with Theorem \ref{2.theorem}, because in DPd 
the eigenvalue $\Lambda_{\tilde X, {\rm min}}$ is smaller than in DPi.

\begin{figure}
\centering
\includegraphics[width=\textwidth]{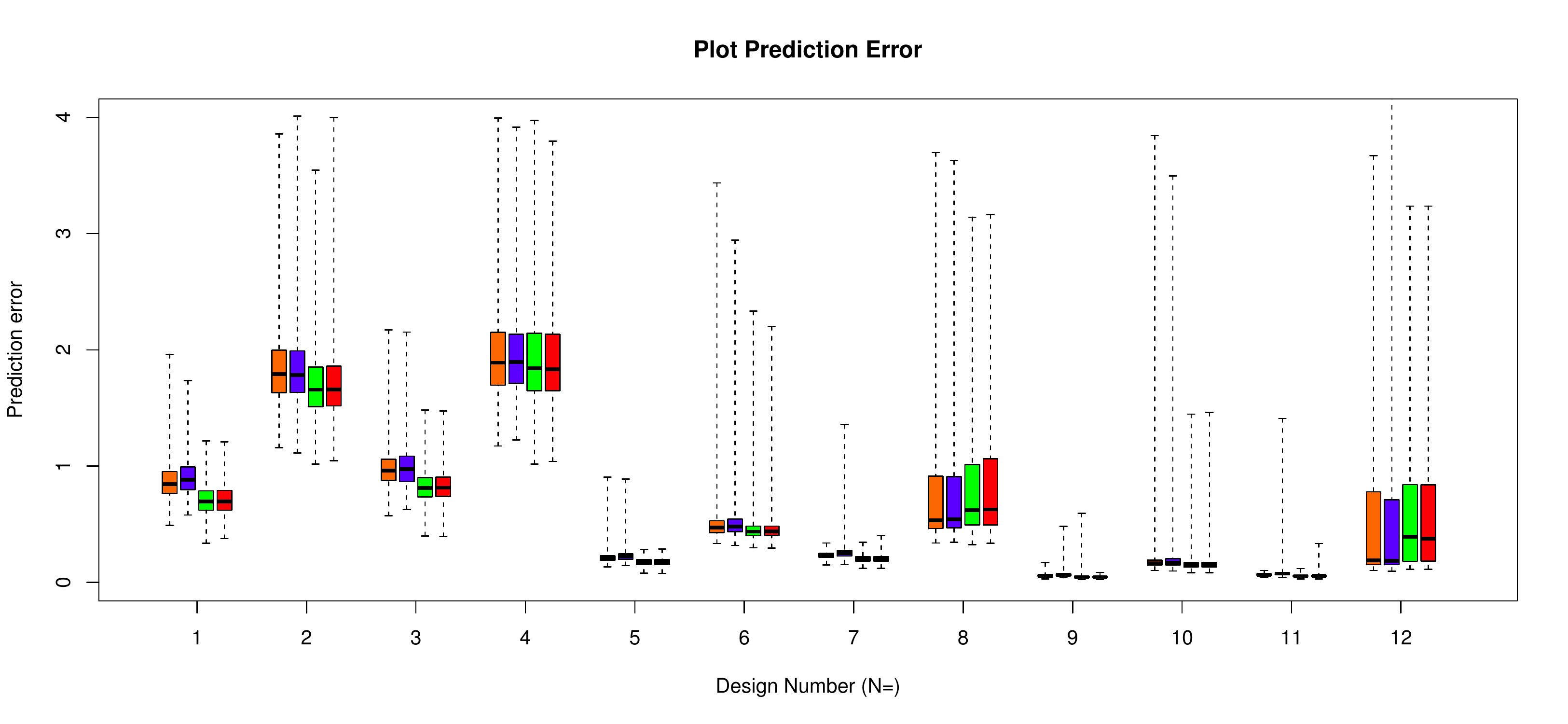}\\
\includegraphics[width=\textwidth]{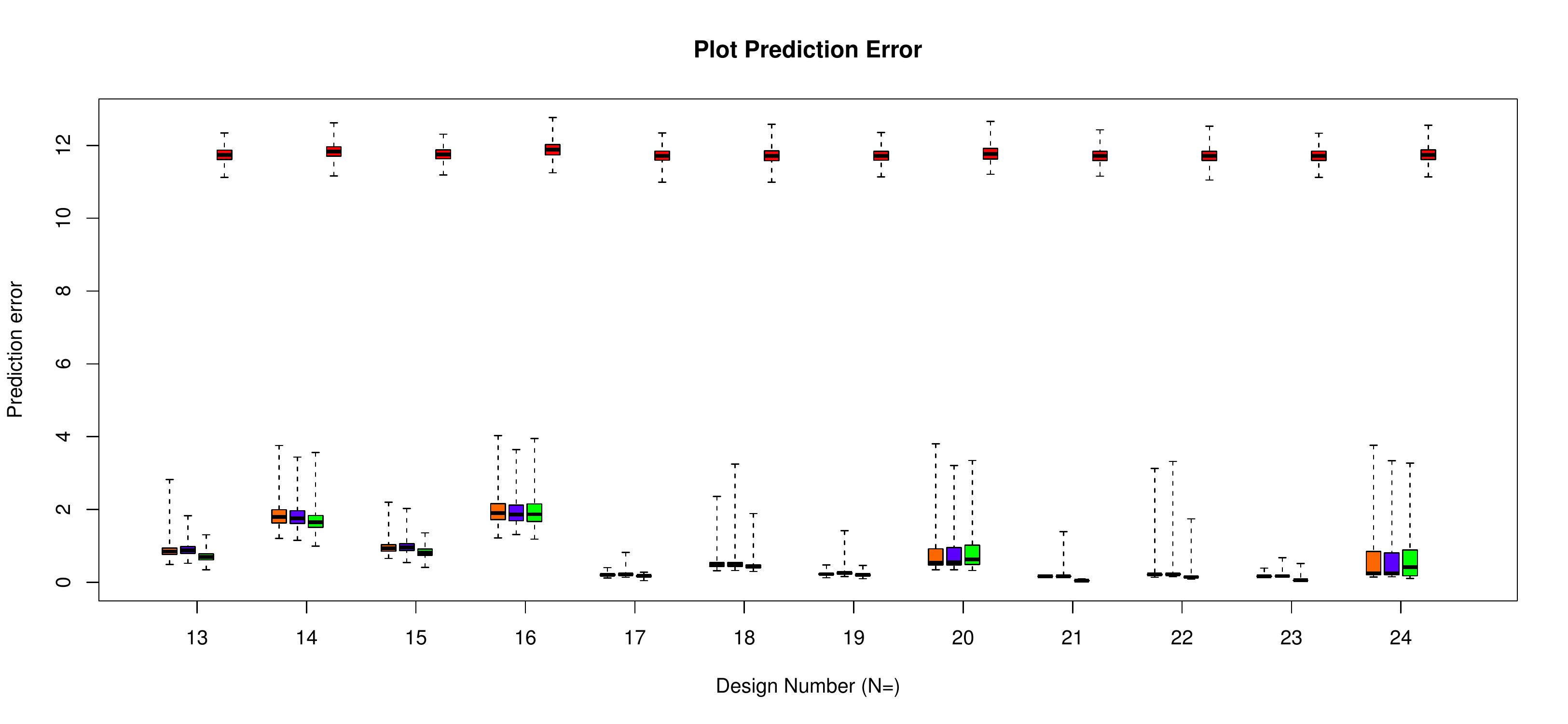}\\
\includegraphics[width=\textwidth]{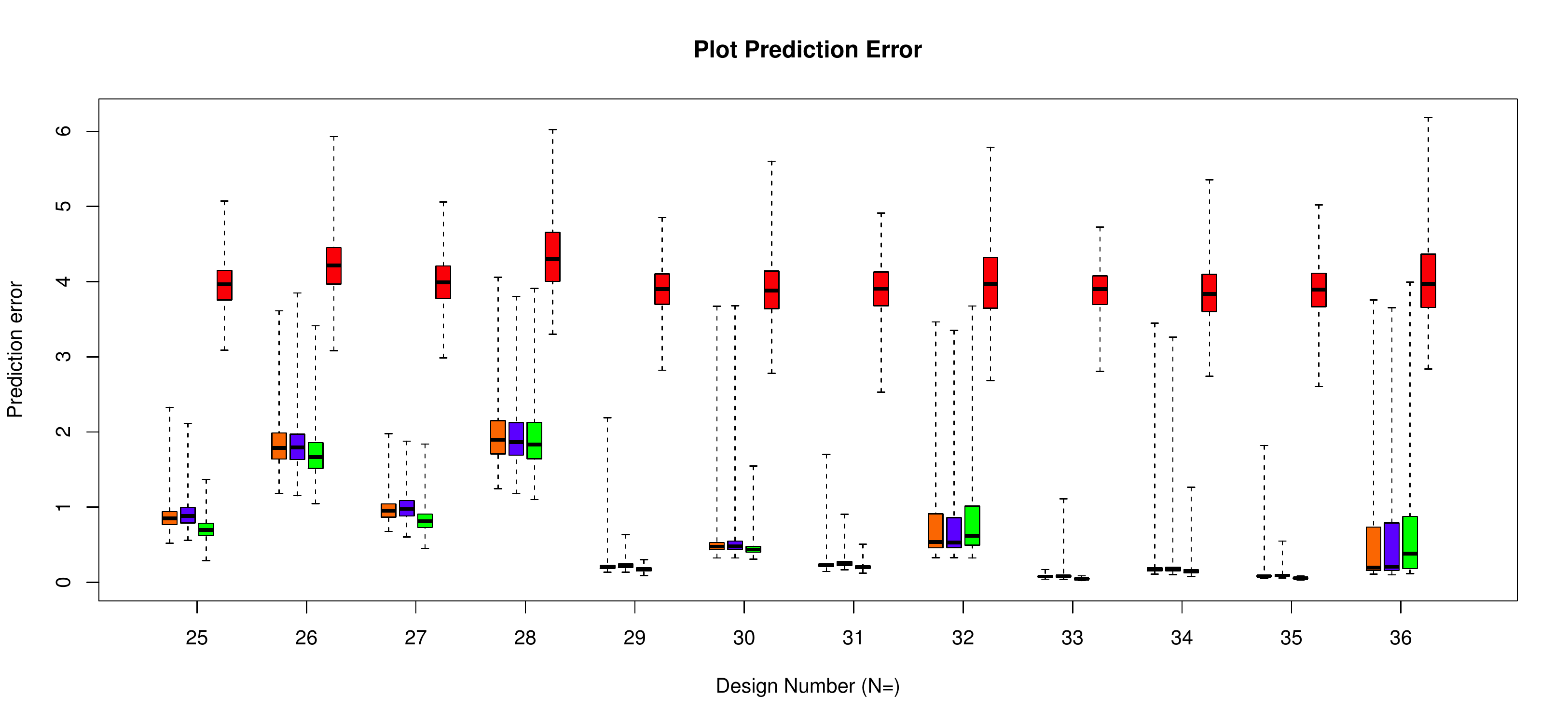}
\caption{Prediction error for designs $N=1,\ldots,36$ with dataset Leukemia. For each of the 36 designs the 
boxplots represent the prediction error for DPi (orange), DPd (blue) LK (green) and LN (red) based on 1000 
replicates.}
\label{Boxplot_Pred}
\end{figure} 
%
\begin{figure}
\centering
\includegraphics[width=\textwidth]{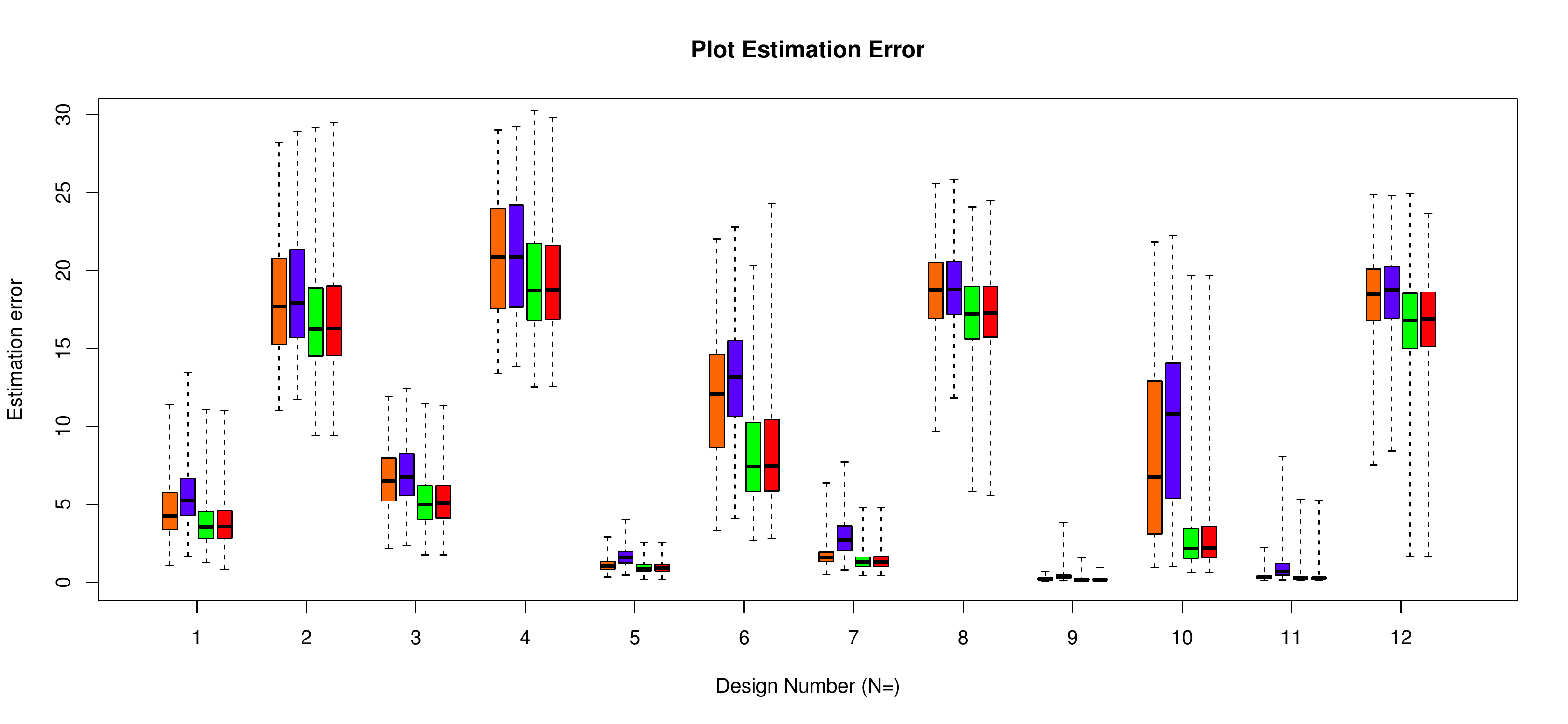}\\
\includegraphics[width=\textwidth]{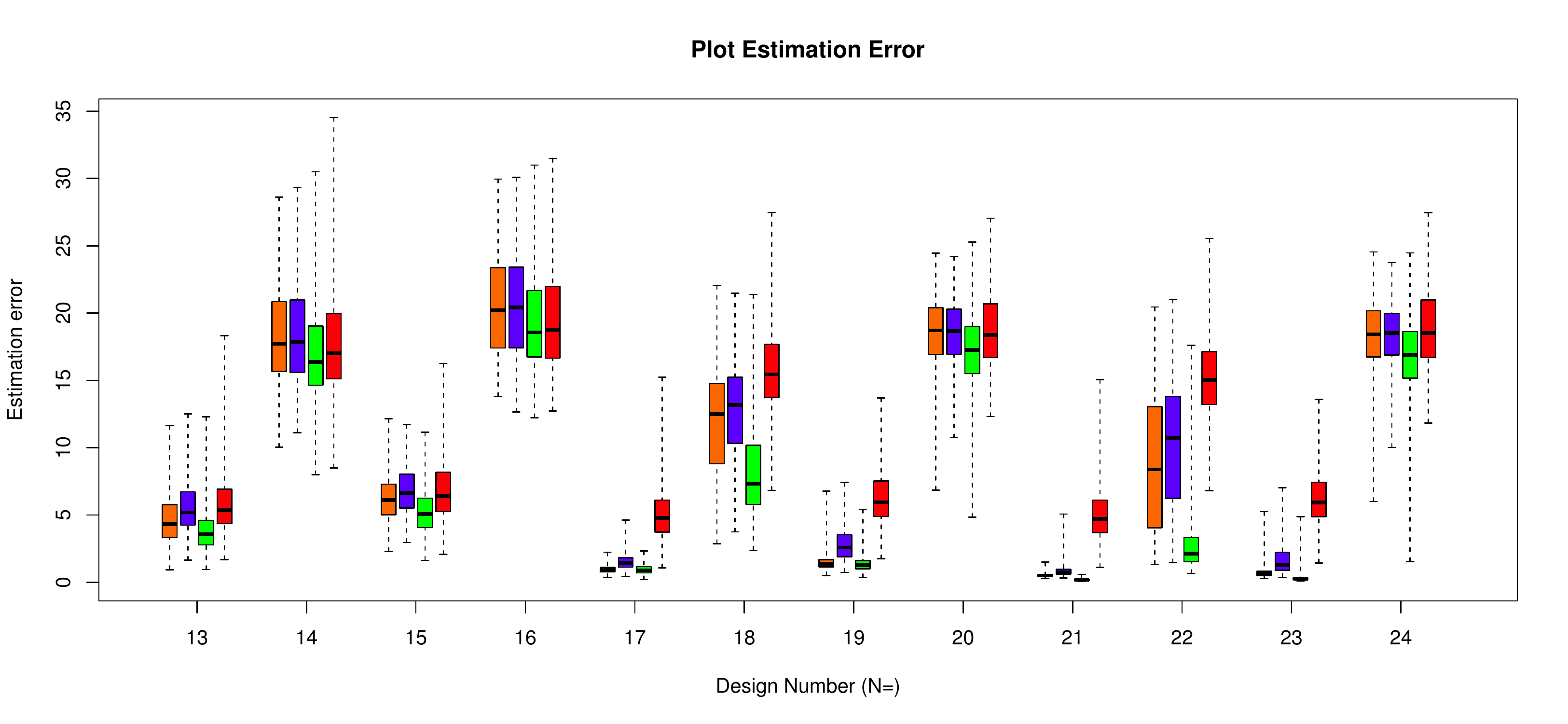}\\
\includegraphics[width=\textwidth]{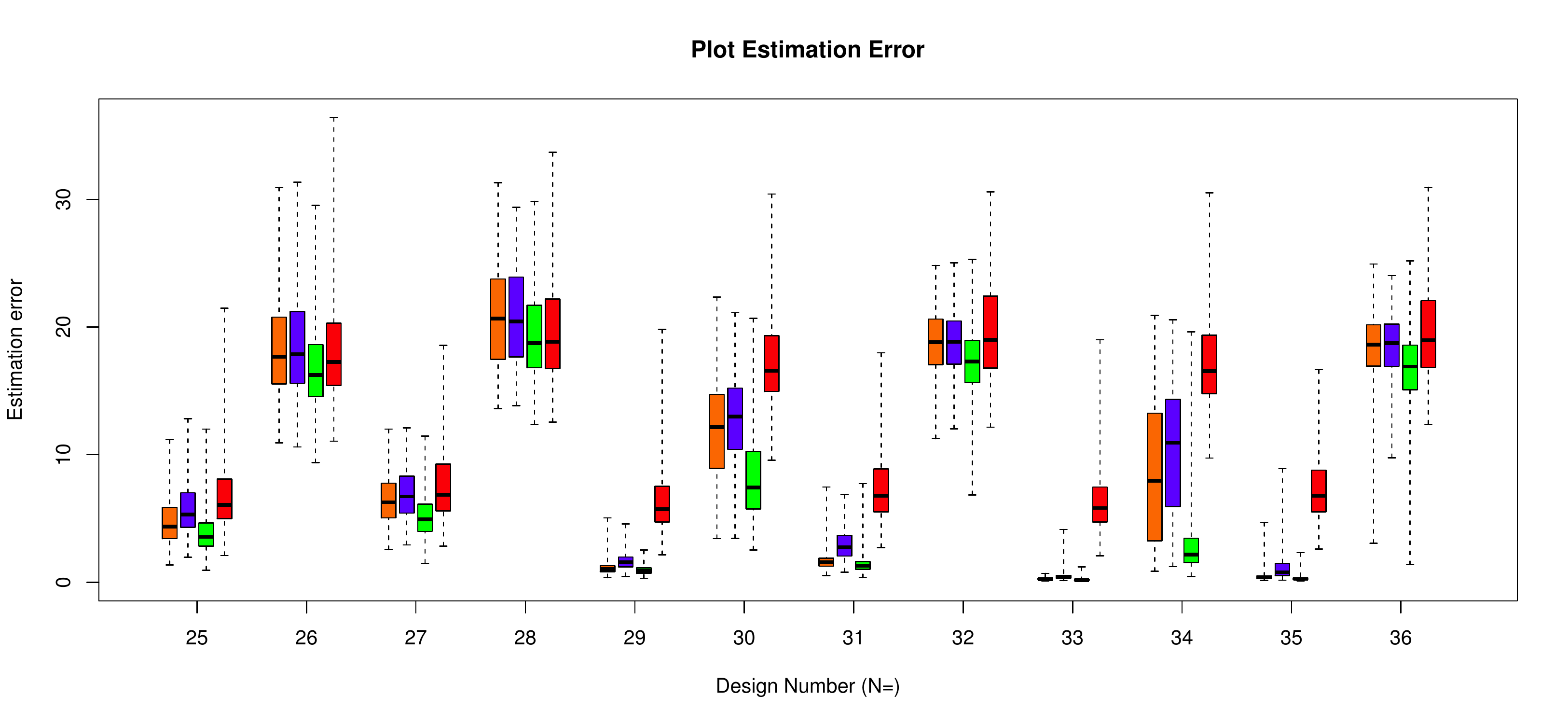}\\
\caption{Prediction error for designs $N=1,\ldots,36$ with dataset Leukemia. For each of the 36 designs the 
boxplots represent the estimation error for DPi (orange), DPd (blue) LK (green) and LN (red) based on 1000 
replicates.}
\label{Boxplot_Estim}
\end{figure} 
%
\subsubsection*{Estimation of the nuisance function:}
We now want to give a closer look at the estimation of the function $g^0$. From the results of the simulation
we can conclude that
\begin{itemize}
 \item The quality of the estimation of the function $g^0$ depends on the performance of $\hat\beta$. In general
better estimation of $\beta^0$ means a better estimation of $g^0$.
 \item If an increase of lSNR improves the quality of the estimation of $\beta^0$, then it also improves
the estimation of $g^0$ (see Figure \ref{gEst.Figure}).
 \item The function choice of $g^0$ does not play crucial role for the quality of the fit.
 \item Even if the estimation of $\beta^0$ (and consequently of $g^0$) are very bad, we can at least have an 
idea
of ``how the original function $g^0$'' looks like (see Figure \ref{gEst.Figure2}).
\end{itemize}

Figure \ref{gEst.Figure} shows the estimation of $g_3^0$ where $p=250$, $s_0=5$ and
the signal to noise ratio varies between 2 and 32. As one can expect, improving SNR makes the estimation of
$g^0$ more accurate.

\begin{center}
\begin{figure}
\centering
\begin{tabular}{ccc}
\includegraphics[width=112pt]{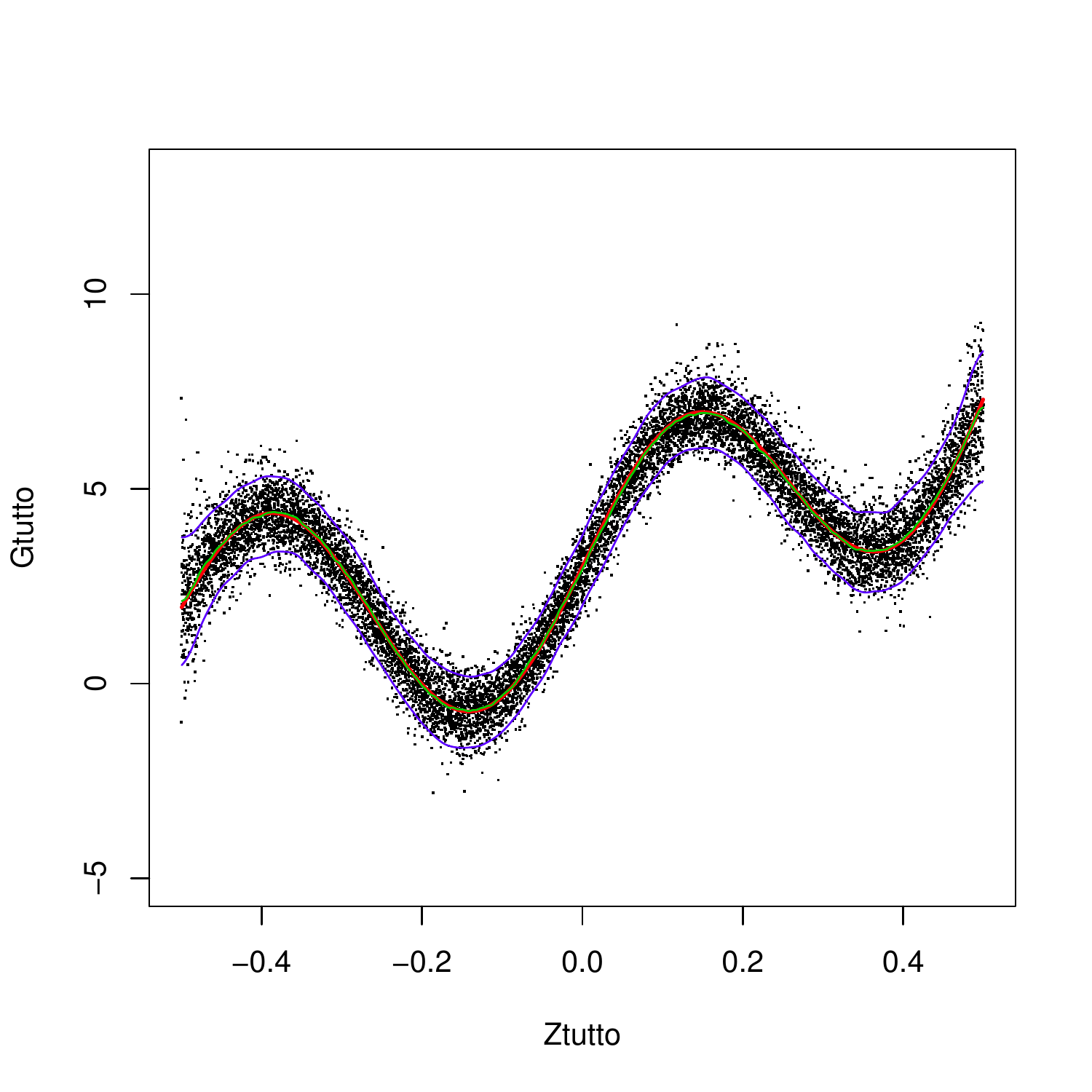}&
\includegraphics[width=112pt]{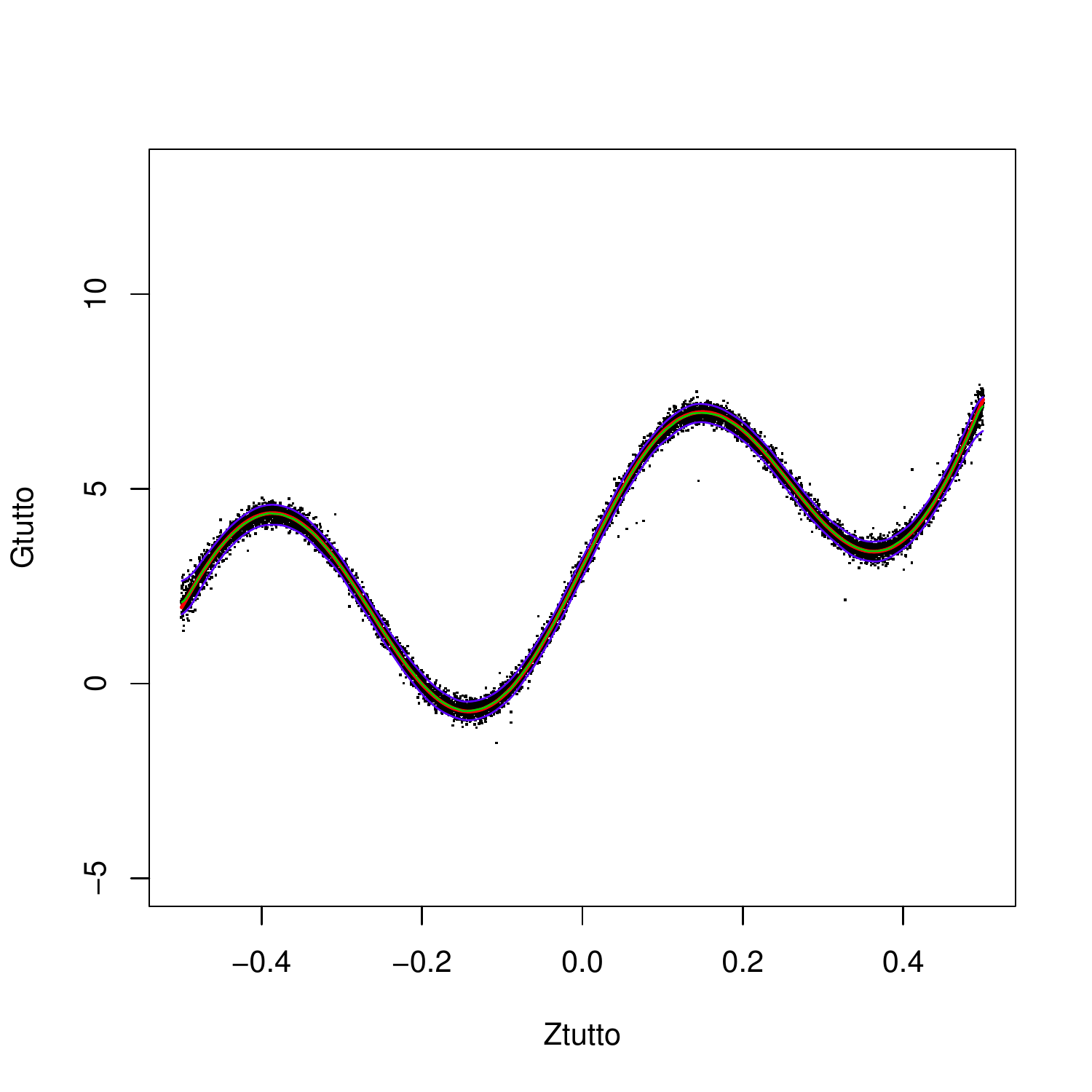}&
\includegraphics[width=112pt]{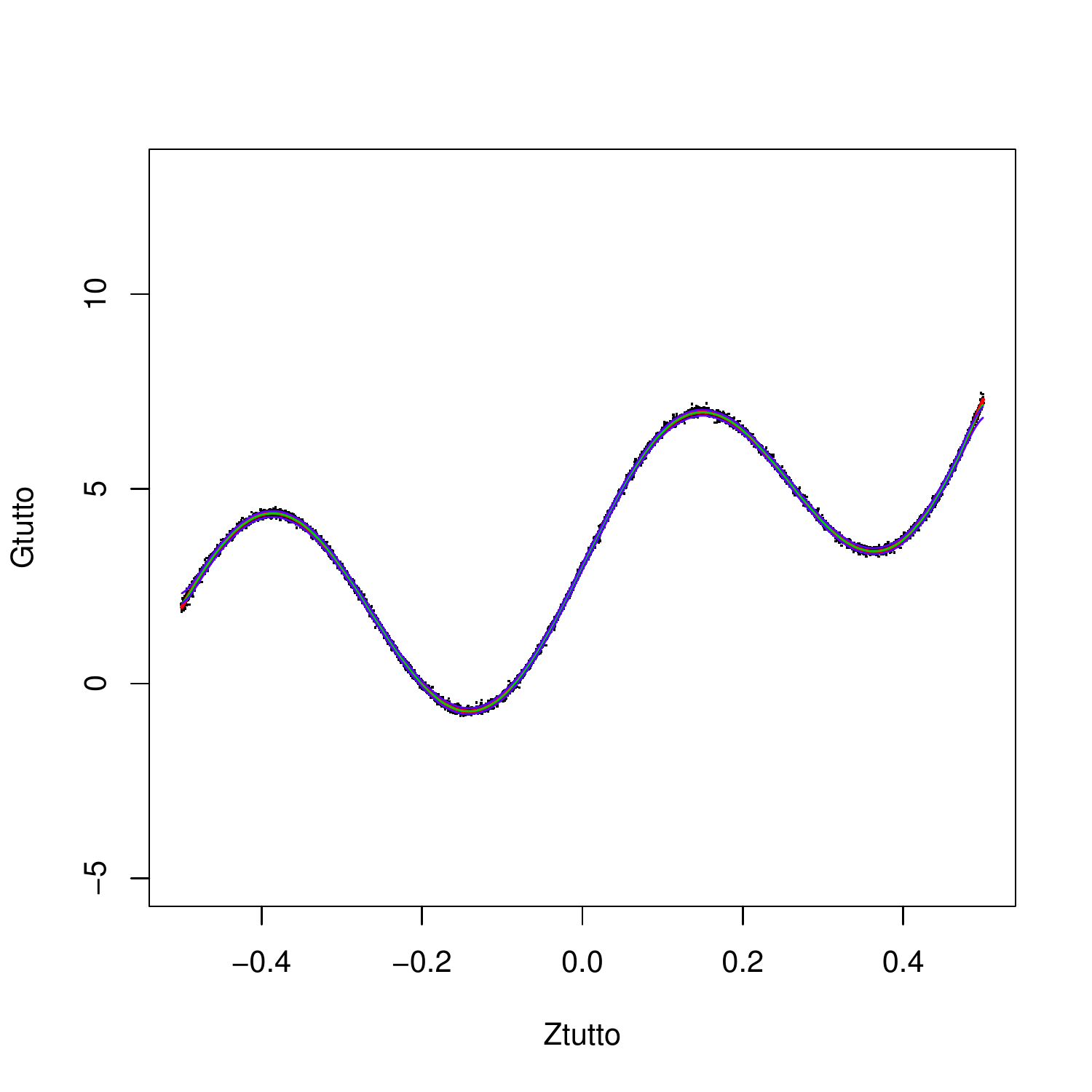}
\\
1a) & 2a) & 3a)\\
&&\\
\includegraphics[width=112pt]{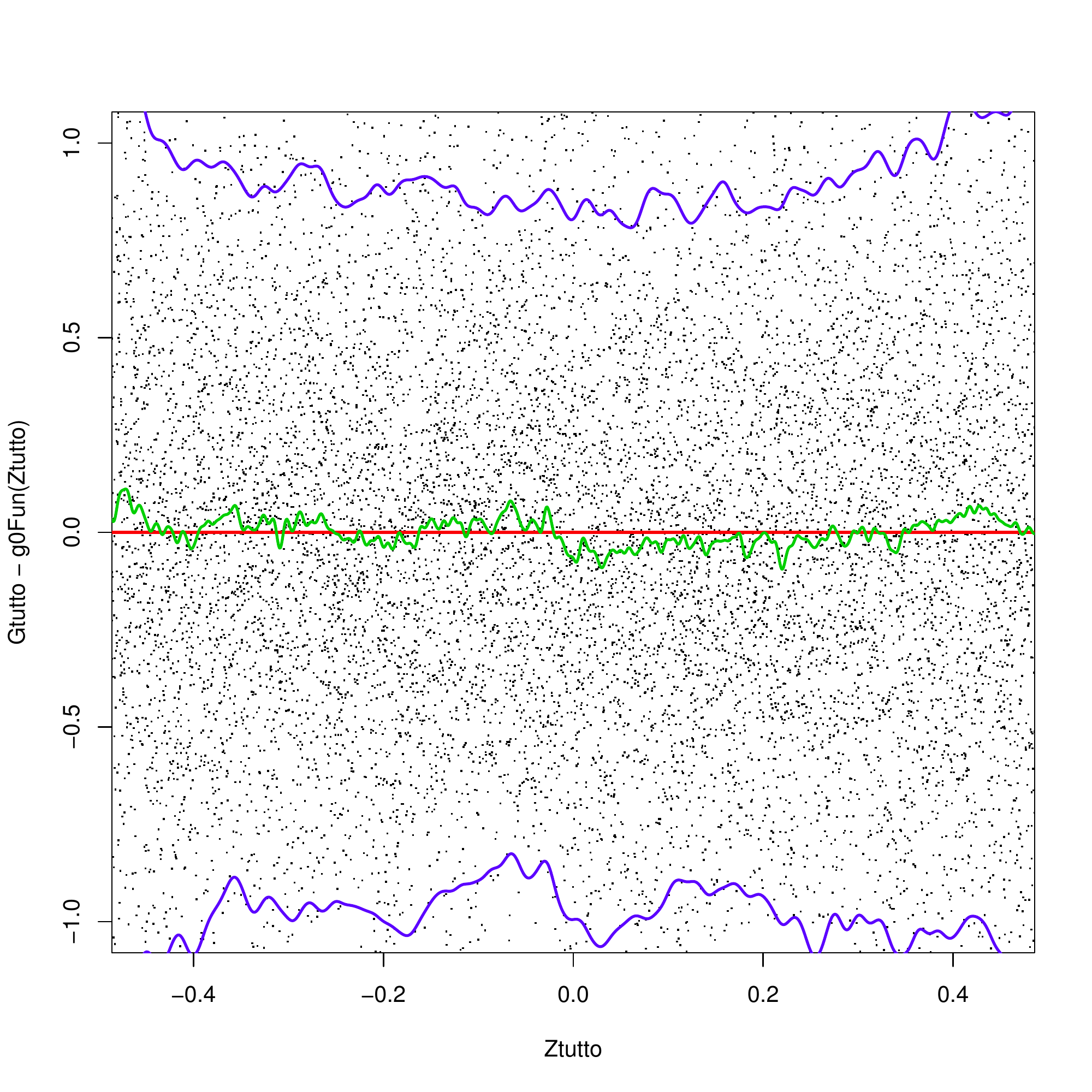}&
\includegraphics[width=112pt]{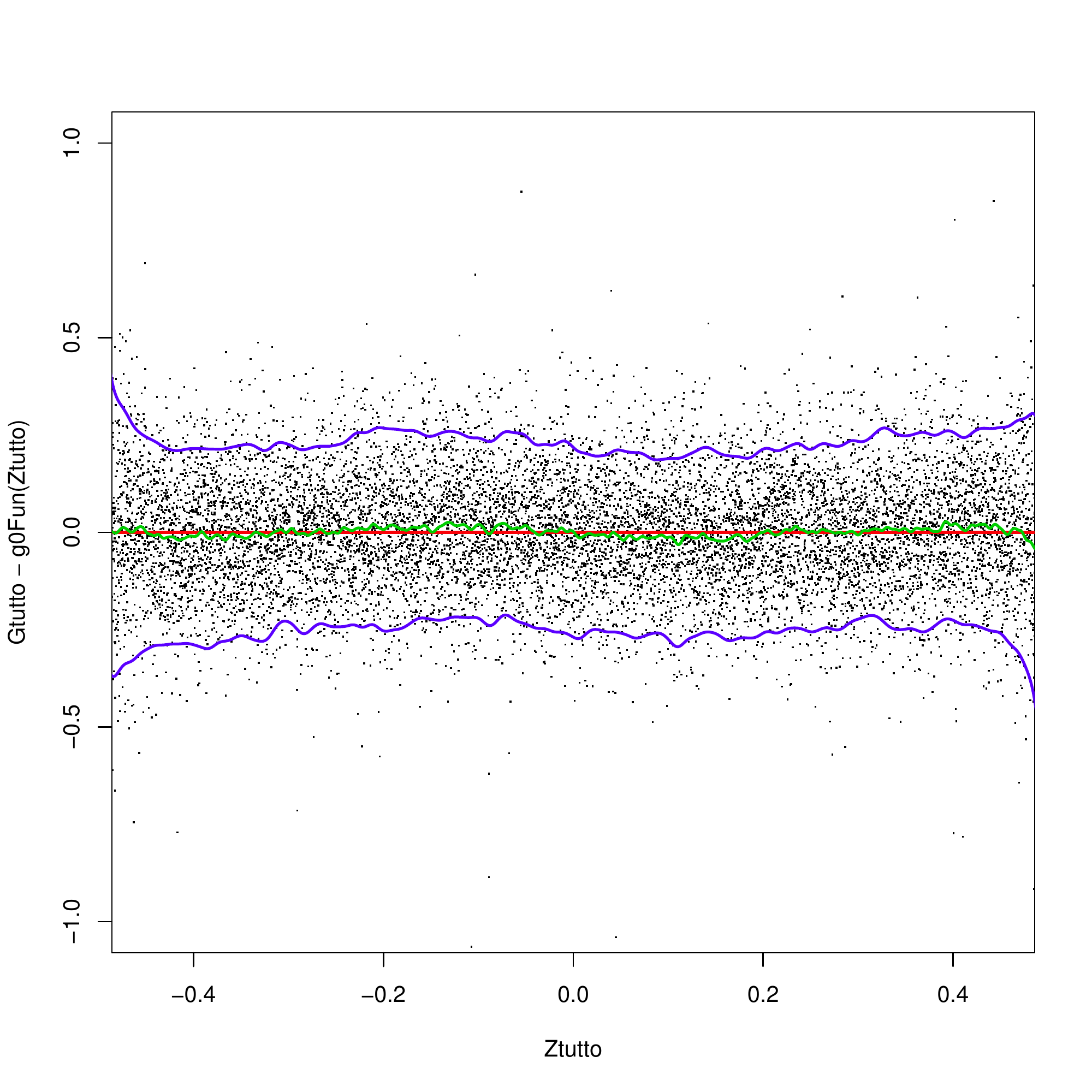}&
\includegraphics[width=112pt]{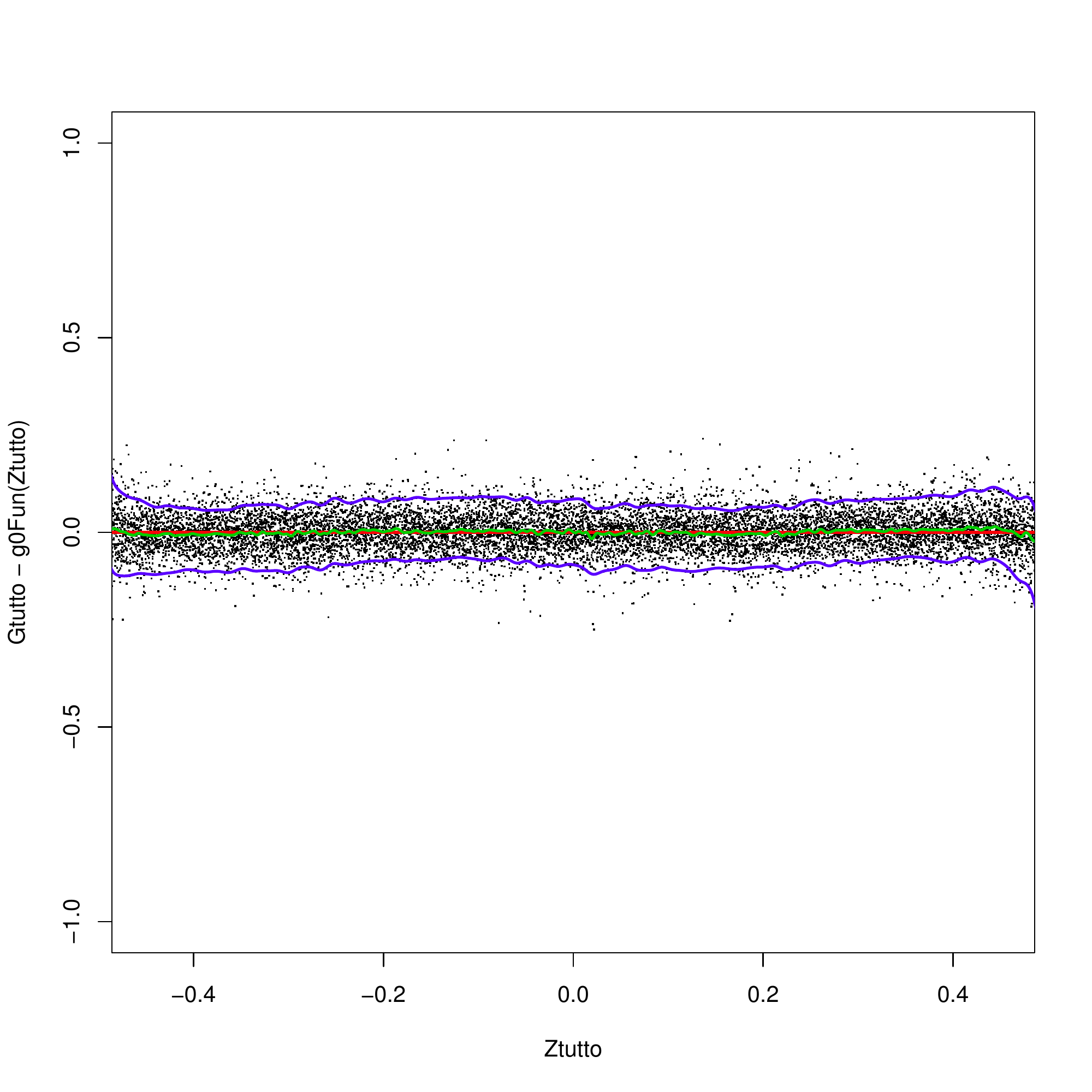}
\\
1b) & 2b) & 3b)\\
\end{tabular}
\caption{Estimation of the function $g^0$ for the designs $N=25,29,33$ respectively in dataset Leukemia. Here 
$p=250$, $s_0=5$ and the nuisance function $g_3$ are kept constant, whether $lSNR$ varies between 2 and 32. In 
the first row the estimated
function is plotted, in the second the residuals. The red line represents the true function. The blue line
corresponds to the mean of the estimated functions, and the blue lines define pointwise a 90\% confidence
interval. \newline An increase of lSNR generally improves the estimation of $g^0$.}
\label{gEst.Figure}
\end{figure} 
\end{center}

\begin{remark}[Estimation quality of $g^0$ for bad working LK method:]
If LK does not properly estimate $\beta^0$, then we can not expect that our method does better. The 
estimation of $g^0$ is then also problematic, but even in this really bad and difficult case, it is possible to 
at least recognise some trends in $g^0$ and have an idea of the order of magnitude of $g^0$. Even if not 
totally satisfying this gives us a basic idea of how $g^0$ looks like.\\
We present as a prime example design number 28 (see Figure \ref{gEst.Figure2}, where $p=1000,~s_0=15$ and 
$SNR=2$. The plot shows that the $10\%$ point wise confidence bounds more or less correspond to $g^0\pm 2.5$.\\
Looking at some randomly chosen results for the estimation of $g^0$, we are still able to somehow have an 
idea on how $g^0$ looks like, even if the estimation is very rude.
\end{remark}

%
%
\begin{figure}
\centering
\begin{tabular}{ccc}
\includegraphics[width=112pt]{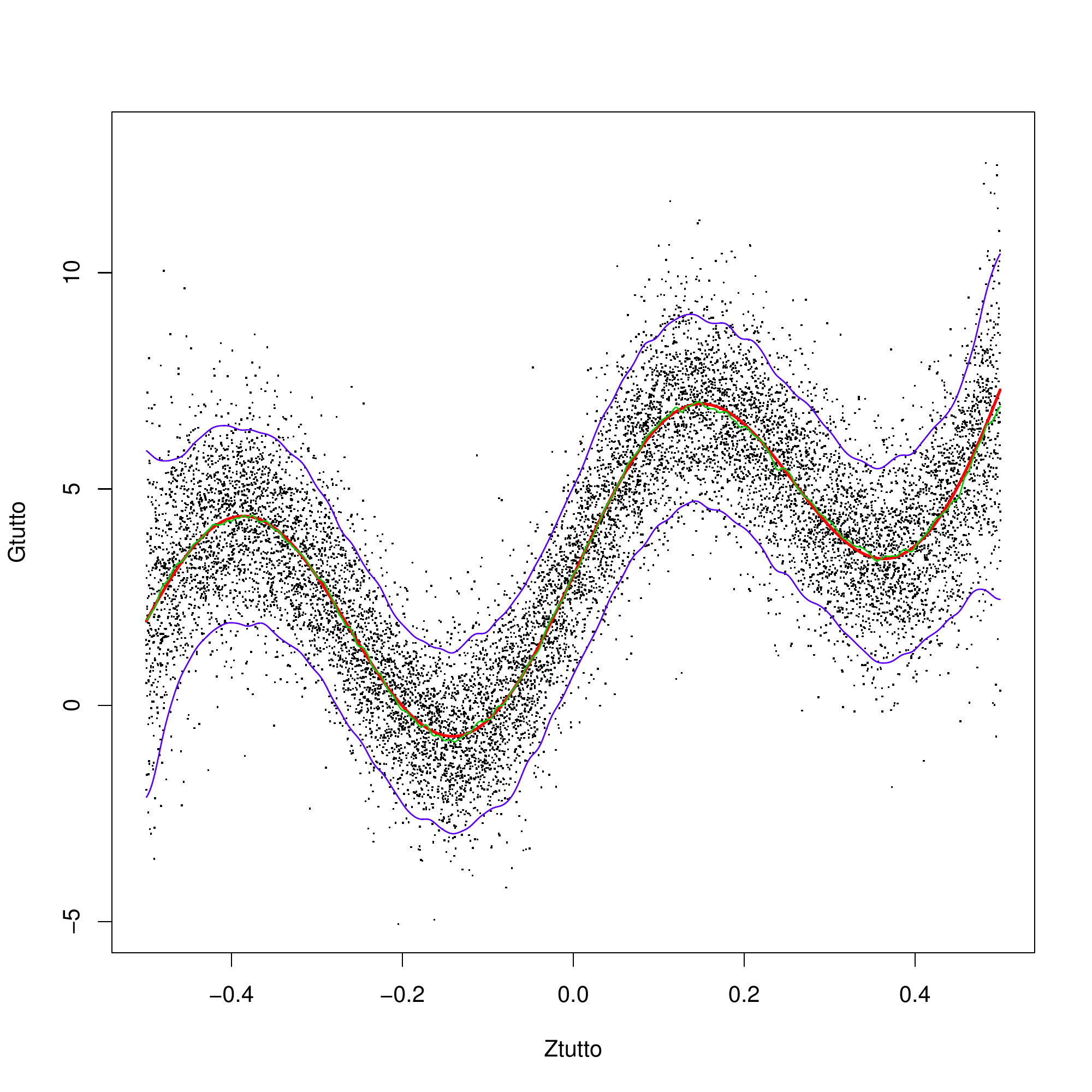}&
\includegraphics[width=112pt]{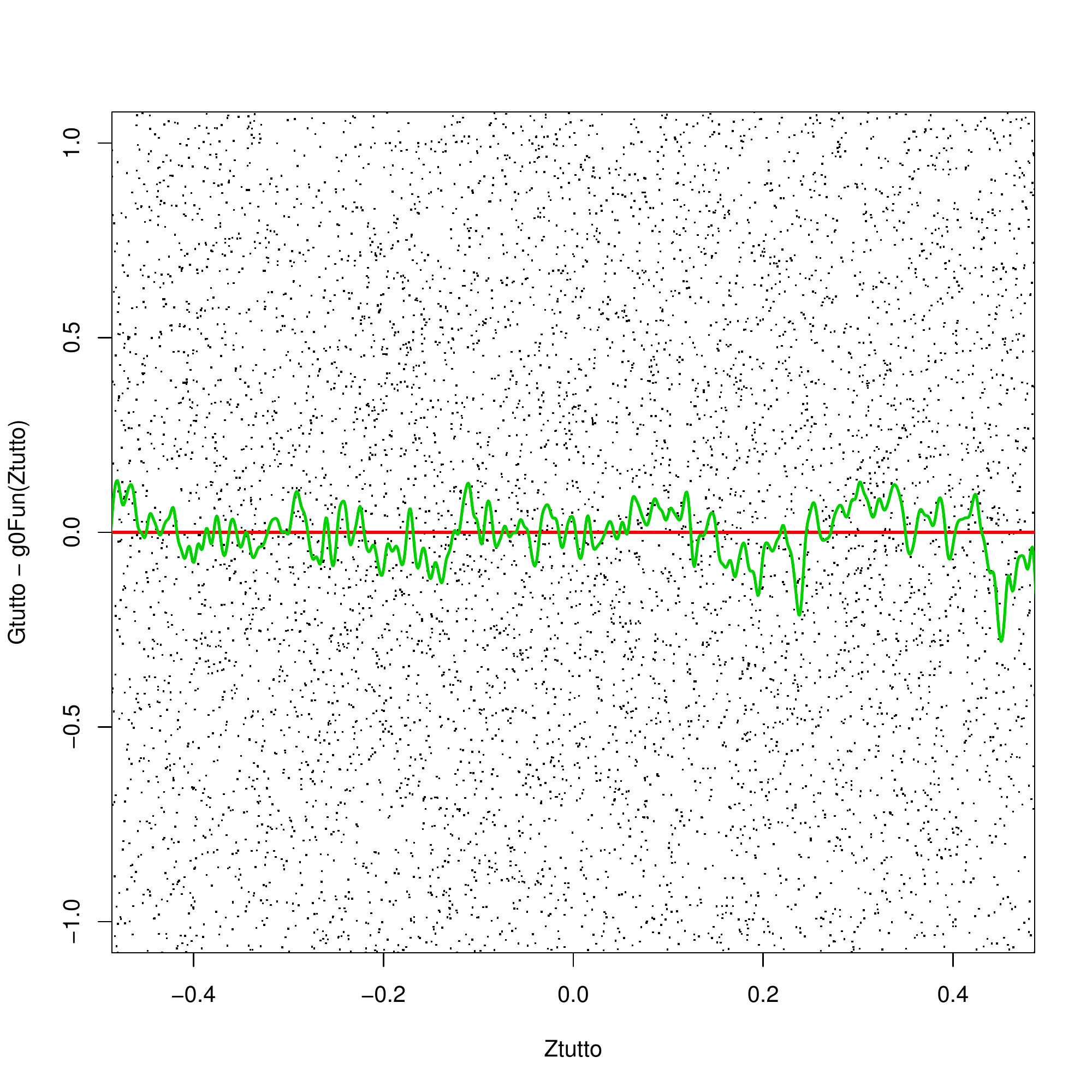}&
\includegraphics[width=112pt]{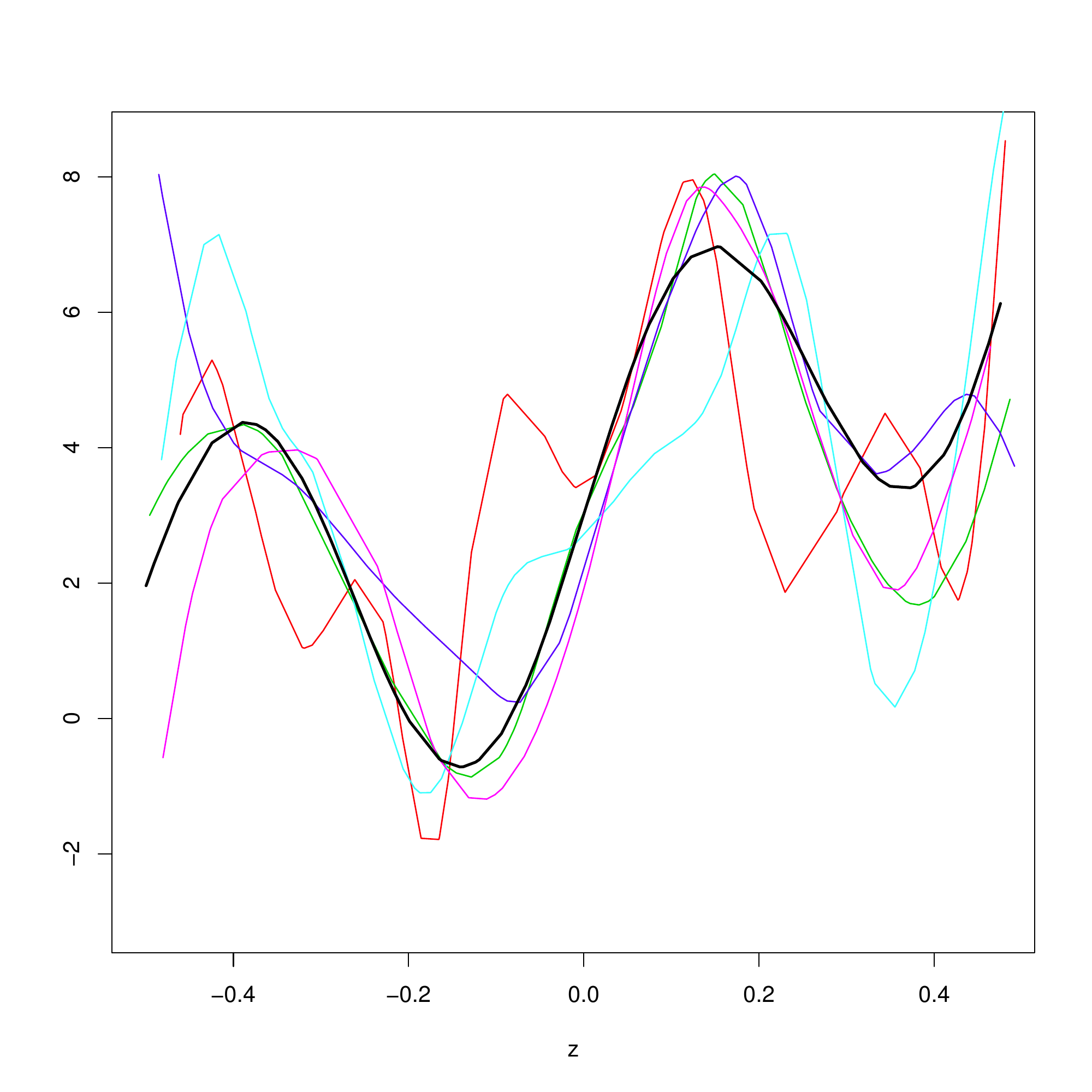}
\\
i) & ii) & iii)\\
\end{tabular}
\caption{Estimation of the function $g^0$ for design $N=28$ ($p=1000,~s_0=15,~SNR=2$).
The first plot shows the true function 
$g^0$ (red), the mean of the estimated function on 1000 replicates (green) and 90\% confidence bounds (blue). 
The coordinates of the black dots are $(z_i,g_i)$.
In the second figure the residuals are plotted.
\newline The third plot represents the function $g^0$ (black) and the estimated functions $\hat g$ for five 
(randomly chosen) replicates.}
\label{gEst.Figure2}
\end{figure} 
%
\subsubsection*{Limit cases:}
Here we briefly discuss, without presenting all numerical results, some extreme cases.\\
\underline{Very large $g^0$:} If the nuisance function completely dominates the linear part the estimation of 
$g^0$ is easier (small relative error) than in the previous cases, but we pay in term of estimation of 
$\beta^0$. For example for $g^0(z):=200g^0_3(z)+400$, even if LK has a TPR of 100\%, our method is on average 
not able to pick out more than 28-32\% of the true positives. This is in any case better than LN, which in this 
case scores similarly as just randomly selecting components.

\vskip .1in
\underline{The tuning parameter $\mu$}
In our design $\mu$ is chosen as $\mu_0:=n^{-2/5}/100$.
The choice of $\mu$ seems to play a marginal role in the estimation quality. 
We tried to repeat the simulations with some different values of $\mu$. Even multiplying or dividing $\mu_0$ by
a factor 10 does not substantially affect our results.\\
Our estimator is then quite robust in $\mu$, but as one can expect extreme values of $\mu$ (like 
$\mu=10000\cdot\mu_0$) leads to problematic estimation of $g^0$.

\vskip .1in
\underline{Very highly correlated covariates}
If $Z$ is strongly correlated with a linear combination of the active part of $X$, then $\Lambda_{\tilde X, {\rm 
min}}$ is (almost) equal to zero. This entails that the estimation is bad, but we can anyway still have 
appreciable prediction results.

\section{Conclusions} As partial linear models have been successfully applied, and
the parametric part can be estimated with parametric rate, it is both from
practical as well as theoretical point of view of interest to extend the model and
theory to the case where the linear part is high-dimensional. As we have shown,
the linear part can be estimated with oracle rates as if the nuisance function were known,
even though the rate for estimating the nuisance function is slower than the oracle rate
for the linear part. 

Our method of proof makes use of empirical process theory, in particular it shows uniform
probability inequalities for empirical $L_2$-norms. It moreover shows bounds for empirical correlations
between transformed observations, uniformly over a class of transformations. 

The simulations show that it indeed may
pay off to include the nuisance function into the model. From computational point of view, the methodology
does not impose additional difficulties when using quadratic penalties for the nuisance part.


\section{Proofs}\label{Proffs.section}
\subsection{Sketch of proofs}
We prove Theorem \ref{1.theorem} in two steps. We first assume that we are under some event $\mathcal{T}$.
Conditionally on $\mathcal{T}$ all conclusions of Theorem \ref{1.theorem} hold with probability $1$ (see Lemma
\ref{Teo_T.lemma}). Then we show that the event $\mathcal{T}$ has large probability, at least
$1-3\exp[-n\deltaTOT^2\mu^2]$ (see Lemmas \ref{classG.lemma}-\ref{T2.lemma}). This ends the proof.
Theorem \ref{2.theorem} supplies stronger results. The proof goes along the lines of the first theorem but uses
the first theorem as starting point.

\subsection{The result assuming empirical process conditions}

%

Let for each $R>0$
$${\cal F} (R):= \{ f: \tau_{\mu,R}(f) \le R \} $$
and
$${\cal T} (\delta_0 , R):=  {\cal T}_1 (\delta_0, R) \cap {\cal T}_2 (\delta_0, R) , $$
where
$${\cal T}_1 (\delta_0 , R) :=  \biggl \{(X,Z)\ :\ \sup_{f \in {\cal F} (R)} \biggl | \|f \|_n^2 - \| f 
\|^2 
 \biggr | \le
\delta_0 R^2  \biggr \} $$
and 
$$ {\cal T}_2 (\delta_0, R) := \biggl \{(X,Z,E)\ :\ |\sup_{f \in {\cal F} (R)} |  E^T f /n | \le 
\delta_0 
R^2  \biggr
\} . $$

\begin{lemma} \label{Teo_T.lemma}
Let $\mu^2\le  \delta_0 R^2 / ( 2 J^2 (g^0)) $. Assume Condition \ref{eigenvalue.condition} and that 
$$  {4  \lambda^2 s_0 \over     \Lambda_{\tilde X, {\rm min}}^2 }  \le \delta_0 R^2 . $$
Then on ${\cal T} (\delta_0 , R)$, it holds that
$\tau( \hat f - f^0) \le R $.
\end{lemma} 
\begin{proof}

Define
$$\tilde f := s \hat f + ( 1- s ) f^0 = X  \tilde \beta  + \tilde g , $$
where
$ \tilde \beta = s \hat \beta + (1- s) \beta^0 $ and $\tilde g = s \hat g + (1-s ) g^0 $ and 
$$s := {R \over R+ \tau ( \hat f - f^0 ) } . $$

We prove this lemma with the following trick. The function $\tilde f$ is defined in such a 
way that $\tau(\tilde f-f^0)\leq R$. Then on $\mathcal{T}$, from the ``basic inequality" below, we get 
$\tau(\tilde 
f-f^0)\leq R/2$. This implies $\tau(\hat f-f^0)\leq R$.
\bigskip \\
 We now present the details of this trick. By convexity and definition of $\hat f$, we know that
$$ \| Y - \tilde f \|_n^2 + \lambda \| \tilde \beta \|_1 + \mu^2 J^2 ( \tilde g) \le$$
$$ s \biggl (  \| Y - \hat f \|_n^2 +  \lambda  \| \hat \beta \|_1 +  \mu^2  J^2  (\hat g) \biggr ) +
(1-s ) \biggl ( \| Y - f^0 \|_n^2 + \lambda  \| \beta^0 \|_1 +  \mu^2  J^2  (g^0)\biggr ) $$
$$ \le  \| Y - f^0 \|_n^2 + \lambda  \| \beta^0 \|_1 +  \mu^2  J^2  (g^0) . $$
This is the ``basic inequality". 
Using that $Y=f^0+E$ rewrite this to
$$\| \tilde f - f^0 \|_n^2 + \lambda \| \tilde \beta \|_1 + \mu^2 J^2 ( \tilde g ) \le
2 E^T ( \tilde f - f^0 ) /n + \lambda \| \beta^0 \|_1 + \mu^2 J^2 (g^0). $$
Now use that
$$\tau ( \tilde f - f^0 ) = s \tau ( \hat f - f^0 ) =  {R \tau (\hat f - f^0 )  \over R+ \tau ( \hat f - f^0 )
}\le R . $$
Hence on ${\cal T} (\delta_0, R)$, we find
$$\| \tilde f - f^0 \|^2 +  \lambda \| \tilde \beta \|_1 + \mu^2 J^2 ( \tilde g ) \le 3 \delta_0 R^2 +
\lambda \| \beta^0 \|_1 + \mu^2 J^2 (g^0) . $$
Using Equation \eqref{f_Orthogonal_decomposition.equation}, the identity
$\|\tilde\beta\|_1=\|\tilde\beta_{S_0^c}\|_1+\|\tilde\beta_{S_0}\|_1$ and then the triangle inequality, we
obtain
$$
\| \tilde X ( \tilde \beta - \beta^0 ) \|^2 + \| h ( \tilde \beta - \beta^0)  + ( \tilde g - g^0 ) \|^2 +
\lambda \| \tilde \beta_{S_0^c } \|_1 + \mu^2 J^2 ( \tilde g )
$$
\be \label{Step_Lemma_Teo_T.equation_1}
\le 3 \delta_0 R^2 + \mu^2 J^2 (g^0)+ \lambda \| \beta^0  - \tilde \beta_{S_0}\|_1 .
\ee
Using the approximation $uv\leq u^2 +\frac{v^2}{4}$ and the assumptions of the
lemma, we get
\be 
\lambda||\tilde \beta_{S_0}-\beta^0||_1&\leq& \lambda \sqrt{s_0}||\tilde \beta_{S_0}-\beta^0|| \leq
\lambda \sqrt{s_0}||\tilde \beta-\beta^0|| \nonumber \\
&\leq& \frac{\lambda \sqrt{s_0}}{\Lambda_{\tilde X,{\rm min}}} ||\tilde X(\tilde\beta-\beta_0)|| \leq
\frac{\lambda^2
s_0}{\Lambda_{\tilde X,{\rm min}}^2}+\frac{||\tilde X (\tilde\beta-\beta_0)||^2}{4}\nonumber \\
&\leq& \frac{\delta_0 R^2}{4}+\frac{||\tilde X (\tilde\beta-\beta_0)||^2}{4} \label{Step_Lemma_Teo_T.equation_2}
\ee
and similarly
\be \label{Step_Lemma_Teo_T.equation_3}
2\lambda ||\tilde \beta_{S_0}-\beta^0||_1\leq \frac{\lambda^2 s_0}{\Lambda_{\tilde
X,{\rm min}}^2}+||\tilde X (\tilde\beta-\beta_0)||^2 \leq \frac{\delta_0 R^2}{4}+||\tilde X 
(\tilde\beta-\beta_0)||^2\ . 
\ee
Note now that using $4uv\leq 4u^2 +v^2$ it yields
$$
\frac{3}{4}J^2(\tilde g-g^0)=\frac{3}{4}\left( J^2(\tilde g)+J^2(g^0)\right)+\frac{3}{2} J(\tilde g,g^0)
$$
$$
\leq \frac{3}{4}\left( J^2(\tilde g)+J^2(g^0)\right) +\frac{1}{4} \cdot 4 \cdot \frac{3}{2}J(\tilde g,g^0)
$$
\be \label{Step_Lemma_Teo_T.equation_4}
\leq \frac{3}{4}\left( J^2(\tilde g)+J^2(g^0)\right) +\frac{9}{4}J^2(g^0)+\frac{1}{4}J^2(\tilde g)
= J^2(\tilde g)+3 J^2(g^0)\ . 
\ee

Adding on both sides of inequality \eqref{Step_Lemma_Teo_T.equation_1}  $\frac{3}{4}\mu^2 J^2(\tilde g-g^0)$ and
using successively Inequalities \eqref{Step_Lemma_Teo_T.equation_2} and \eqref{Step_Lemma_Teo_T.equation_4} this 
 
leads to
$$
\|\tilde X ( \tilde \beta - \beta^0 ) \|^2 + \| h ( \tilde \beta - \beta^0)  + ( \tilde g - g^0 )
\|^2 +\frac{3}{4} \mu^2 J^2(\tilde g-g^0)
$$
$$
\leq 3\delta_0 R^2+ 4\mu^2 J^2(g^0)+\frac{\delta_0 R^2}{4} +\frac{1}{4}\| \tilde X (\tilde \beta - \beta^0 )\|^2
- \lambda \| \tilde \beta_{S_0^c } \|_1 \ . 
$$
By the assumption $\mu^2 \le \delta_0 R^2 / (2J^2 (g^0))$ we obtain
$$
\frac{3}{4}\|\tilde X ( \tilde \beta - \beta^0 ) \|^2 + \| h ( \tilde \beta - \beta^0)  + ( \tilde g - g^0 )
\|^2 +\frac{3}{4} \mu^2 J^2(\tilde g-g^0)
$$
$$
\leq 3\delta_0 R^2+ 2\delta_0 R^2+\frac{1}{4}\delta_0 R^2\ . 
$$
Hence, by Remark \ref{OrthogonalPorj.remark}
$$
||\tilde f-f^0||^2 + \mu^2 J^2(\tilde g-g^0) \leq 
7\delta_0 R^2.
$$
 Therefore
 $$ \left( \| X ( \tilde \beta - \beta^0)  + ( \tilde g - g^0 ) \|^2 + \mu^2 J^2 ( \tilde g ) \right)
^{1/2}\leq \left( 7 \delta_0\right)^{1/2}  R.$$
On the other hand, adding now $\lambda \| \tilde \beta_{S_0} - \beta^0 \|_1$ on both sides of 
\eqref{Step_Lemma_Teo_T.equation_1} leads to
 $$
\| \tilde X ( \tilde \beta - \beta^0 ) \|^2 + \| h ( \tilde \beta - \beta^0)  + ( \tilde g - g^0 ) \|^2 +
\lambda \| \tilde \beta - \beta^0  \|_1 + \mu^2 J^2 ( \tilde g ) 
$$
$$
 \leq 3 \delta_0 R^2 + \mu^2 J^2 (g^0)+ 2 \lambda \| \beta^0  - \tilde \beta_{S_0}\|_1
$$
$$
\leq \frac{15}{4} \delta_0 R^2 + \| \tilde X( \tilde \beta - \beta^0 ) \|^2 
$$
where in the last inequality the assumption $\mu^2 \le \delta_0 R^2 / (2 J^2 (g^0))$ and Inequality
\eqref{Step_Lemma_Teo_T.equation_3} were applied.\\
We now have:
$$
\lambda \| \tilde \beta - \beta^0  \|_1 \leq \frac{15}{4} \delta_0 R^2 
$$
and
$$
\frac{\lambda \| \tilde \beta - \beta^0  \|_1}{R\sqrt{\delta_0/2}}\leq \frac{15}{4}\sqrt{2 \delta_0}R~.
$$

Then
$$
\tau( \tilde f - f^0 ) \leq  \left( 7 \delta_0\right)^{1/2}  R+\frac{15}{4}\sqrt{2 \delta_0}R\leq
\frac{\sqrt{253\delta_0}}{2}R=\frac{R}{2} ,
$$
Thus we have 
$$
\tau( \tilde f - f^0 )=s \tau( \hat f - f^0 )=  \frac{R\tau( \hat f - f^0 )}{R+\tau( \hat f - f^0 )}\leq
\frac{R}{2}~.
$$
Hence
$$
\tau( \hat f - f^0 )\leq R
$$
\end{proof}

\subsection*{Technical tools}
The main goal of the following lemmas is to prove that the set $\mathcal{T}$ has large probability.\\
The next lemma is along the lines of Theorem 3.1 in \cite{gine2006concentration}. 

\begin{lemma} \label{classG.lemma} Let ${\cal G} (\tilde{R},\tilde{K})$ be a class of functions $g$ on
$\R^d$, uniformly
bounded by $\tilde{K}$,
and satisfying  
$\|g \| \le \tilde{R} $. Let $\{ \varepsilon_i \}_{i=1}^n $ be a Rademacher sequence independent of $\{ z_i
\}_{i=1}^n $, 
that is $\varepsilon_i$ are $i.i.d.$ and take values $\pm1$ with probability $1/2$ each. Suppose that for
some increasing strictly convex function $G$, it holds that
$$\EE \sup_{g \in {\cal G}(\tilde{R},\tilde{K}) }   \biggl | {1 \over n} \sum_{i=1}^n g_i \varepsilon_i  \biggr
| \le
\EE G^{-1} ( \tilde{R}_n^2 )/ \sqrt n  , $$
where $$\tilde{R}_n^2 = \sup_{g \in {\cal G} (\tilde{R},\tilde{K}) } \| g \|_n^2 . $$
Let $H$ be the convex conjugate of $G$. 
Then for $\tilde{R}^2 \ge H(16 \tilde{K} / \sqrt n) / 2 $, we have
$$\EE G^{-1} (\tilde{R}_n^2)   
  \le G^{-1} ( 4 \tilde{R}^2 )  $$
  and
$$
\EE \sup_{g \in {\cal G} (\tilde{R},\tilde{K}) } \biggl | \| g \|_n^2 - \| g\|^2 \biggr | \le 8\tilde{K} G^{-1}
(4 \tilde{R}^2 ) /\sqrt n .
$$
 
\end{lemma}

\begin{proof}
By symmetrization we have (see e.g. \cite{vanderVaart:96}), 
$$ \EE \sup_{g \in {\cal G} (\tilde{R},\tilde{K}) } \biggl | \| g \|_n^2 - \| g\|^2 \biggr | = \EE \sup_{g \in
{\cal G}(\tilde{R},\tilde{K})}
\biggl | \frac{1}{n}\sum_{i=1}^n \left(g_i^2-\EE g_i^2 \right) \biggr | $$
$$\le 2 \EE  \sup_{g \in {\cal G} (\tilde{R},\tilde{K}) } \biggl | {1 \over n} \sum_{i=1}^n g_i^2 \varepsilon_i
\biggr | $$
$$ \le 8 \tilde{K} \EE  \sup_{g \in {\cal G} (\tilde{R},\tilde{K}) } \biggl | {1 \over n} \sum_{i=1}^n g_i
\varepsilon_i
\biggr | ,$$
where in the last inequality we used the contraction inequality of \cite{Ledoux:91} for Lipschitz loss function 
(see Theorem 14.4 of
\cite{BvdG2011}).
But by assumption:
\be \label{classG_Proof.equation_2}
\EE  \sup_{g \in {\cal G} (\tilde{R},\tilde{K}) } \biggl | {1 \over n} \sum_{i=1}^n g_i \varepsilon_i \biggr |
\le
\EE G^{-1} ( \tilde{R}_n^2 ) / \sqrt n .
\ee
Thus we have for $\tilde{u}:= \EE G^{-1} ( \tilde{R}_n^2 ) $,
$$\EE \tilde{R}_n^2 \le \sup_{g \in {\cal G} (\tilde{R},\tilde{K}) } ||g||^2+ \EE \sup_{g \in {\cal G}
(\tilde{R},\tilde{K})}\Big|||g||_n^2-||g||^2\Big|
\le \tilde{R}^2 + 8\tilde{K} \tilde{u}/ \sqrt n . $$
On the other hand, by Jensen's inequality,
$$\EE \tilde{R}_n^2 = \EE G( G^{-1} (\tilde{R}_n^2)) \ge G \biggl ( \EE G^{-1} (\tilde{R}_n^2)\biggr ) = 
G(\tilde{u})
.$$
Let 
$$
H(v):=\sup_u\{ vu-G(u)  \}
$$
be the convex conjugate of $G$, then we have
$$G(\tilde{u}) \le \tilde{R}^2 + 8\tilde{K}\tilde{u}/ \sqrt n \le \tilde{R}^2 + G(\tilde{u})/2 + H ( 16 
\tilde{K} / \sqrt n ) / 2 . $$
and by the assumption
$$ G(\tilde{u}) \le 2 \tilde{R}^2 + H (16 \tilde{K}/\sqrt n ) \le 4 \tilde{R}^2 . $$
In other words, we showed that
$$G  \biggl ( \EE G^{-1} ( \tilde{R}_n^2 ) \biggr ) \le 4 \tilde{R}^2 . $$
$G$ is monotone, so 
\be \label{classG_Proof.equation_1}
\EE G^{-1} ( \tilde{R}_n^2 )  \le G^{-1} ( 4 \tilde{R}^2 ) .
\ee
\end{proof}

\begin{lemma}\label{classG.corollary}
Assume the conditions of Lemma \ref{classG.lemma}.  Then, for all $\tilde{R}^2 \ge H( 16 \tilde{K}/
\sqrt n)/2$, 
$$ \EE \sup_{g \in {\cal } (\tilde{R},\tilde{K}) } \biggl | {1 \over n } \sum_{i=1}^n g_i \epsilon_i  \biggr |
\le
G^{-1} (4\tilde{R}^2 ) / \sqrt n , $$
and for all $\delta > 0$, 
$$ \EE \sup_{g \in {\cal G} (\tilde{R}, \tilde{K})}  \biggl | \| g \|_n^2 - \| g \|^2 \biggr | \le 4 \delta
\tilde{R}^2 + \delta H
\biggl ( { 8 \tilde{K} \over \delta  \sqrt n }  \biggr ) . $$
\end{lemma}

{\bf Proof.}
Along the lines of Lemma \ref{classG.lemma}, starting from \eqref{classG_Proof.equation_2} we get
$$
\EE \sup_{g \in {\cal G} (\tilde{R},\tilde{K}) } \biggl | \| g \|_n^2 - \| g\|^2 \biggr | \le 8\tilde{K} 
\tilde{u}
/\sqrt n .
$$
The right hand side of the above inequality can be further approximate by
$$
8\tilde{K} \tilde{u} /\sqrt n \leq \delta G(\tilde{u}) +\sup_{ u} \left\lbrace \frac{ 8 \tilde{K} u}{ \sqrt n
}-\delta
G( u)\right\rbrace \leq 4 \delta \tilde{R}^2  +\delta H \biggl ( { 8 \tilde{K} \over \delta  \sqrt n } 
\biggr )~.
$$

\hfill $\sqcup \mkern -12mu \sqcap$

\subsection{The set ${\cal T}_1 (\delta_0 , R) $} 

\begin{lemma} \label{empiricalG.lemma} 
Let for a fixed $\beta$
$$R_n^2 := \sup_{g: \ X\beta + g  \in {\cal F}(R) } \| g \|_n^2 . $$
Assume Conditions \ref{eigenvalue.condition} and \ref{entropy.condition}, that $\mu \le R $ and
\eqref{Assumption_EV_Lambda.equation}, then 
$$\EE \sup_{g: \ X\beta + g  \in {\cal F}(R) }   \biggl | {1 \over n} \sum_{i=1}^n g_i \varepsilon_i  \biggr |
\le
\EE G^{-1} ( R_n^2 )/ \sqrt n  , $$
where $\varepsilon_1, \ldots ,\varepsilon_n$ is a Rademacher sequence independent of $\{ z_i \}_{i=1}^n $, and
where
\be \label{DefG.equation}
G (u) :=   \kappaUnoG  (R/\mu)^{-{2  \over 2m-1}} u^{4m \over 2m-1} , \ u > 0 ,
\ee
with $\kappaUnoG$ a constant depending only on $A$ and $m$.
Moreover, if in addition Condition \ref{penalty.condition} holds, we have for a constant
$\kappaUno$
depending only on $A$, $m$, $\Lambda$ (defined in \eqref{Assumption_EV_Lambda.equation}) 
and $\KG$, and for $\mu^2 \ge \kappaUno n^{-{2m \over 2m+1}} $,

$$\EE  \sup_{g: \ X\beta + g  \in {\cal F}(R) }   \biggl | {1 \over n} \sum_{i=1}^n g_i
\varepsilon_i \biggr | 
\le \kappaUno { R^{1 - {1 \over 2m}}  \over \sqrt n }  [ { R / \mu }]^{1 \over 2m}  $$

and for $\mu^2\geq (\kappaUno \vee \deltaUno^{-\frac{4m}{2m+1}} ) n^{-{2m \over 2m+1}}$
$$
\EE  \sup_{g: \ X\beta + g  \in {\cal F}(R) }  \biggl | \| g \|_n^2 - \| g \|^2 \biggr | \le 8\KG 
\kappaUno\deltaUno R^2. 
$$
\end{lemma} 

Suitable choices for $\deltaUno$, $\kappaUno$ and $\kappaUnoG$ are
\be \label{deltaUno.equation}
\deltaUno\leq \left( 8\CC \sqrt{A} \KG \right)^{-1} \left[ \frac{2m+1}{2\Lambda^2 (2m-1)} 
\right]^{-\frac{2m+1}{4m}}
\ee
\be \label{kappaUno.equation}
\kappaUno := \CC\sqrt{A}\frac{2m}{2m-1} \left( 2\Lambda \right)^{\frac{2m-1}{2m+1}}
\ee
\be \label{kappaUnoG.equation}
\kappaUnoG:=\left( \CC\sqrt{A}\frac{2m}{2m-1}\right)^{-\frac{4m}{2m-1}}
\ee
and results from \eqref{kappaUno1.equation}, \eqref{kappaUno2.equation} and \eqref{KappaUnoGCalcolo.equation} 
respectively. Here, $\CC$ is an appropriate universal constant (coming from Dudley's inequality).

Note that the proposed function $G$ is strictly convex, positive and increasing as required in Lemma
\ref{classG.lemma} and that the assumption $\mu\leq R$ is less restrictive than
Assumption \eqref{condition_mu_Thm.equation}, as required in Theorem \ref{1.theorem}.

\begin{proof}
If $X \beta + g \in {\cal F} (R) $, we have by Remark \ref{OrthogonalPorj.remark}
$\| h\beta +g \|^2  \le R^2 $
as well as
$\| \tilde X \beta \|^2 \le R^2 $.
Hence,
$\| \beta \|_2^2 \le {R^2 / \Lambda_{\tilde X , {\rm min}}^2 } $,
so that
$\| h \beta \|^2 \le R^2 { \Lambda_{h, {\rm max}}^2 / \Lambda_{\tilde X , {\rm min}}^2 } $.
Thus
$$\| g\| \le \| h \beta + g \| + \| h \beta \| \le \left (   1+  { \Lambda_{h, {\rm max}} \over \Lambda_{\tilde
X , {\rm min}} } \right ) R := \Lambda R~.$$
and
$$ \left \| { \mu g \over R } \right \|^2 \le \mu^2 \Lambda^2  \leq 1. $$
Furthermore, $\tau(f)\leq R$ implies
$$J^2 (g) \le R^2 / \mu^2. $$
Hence
$$J \left ( \frac{\mu g}{R} \right )  \le 1 , $$

By Condition \ref{entropy.condition} we have
\bes
&&{\cal H} ( u, \{  g : \ X \beta + g \in {\cal F} (R)\}  , \|  \cdot \|_{\infty}  ) \\
&\le& {\cal H}\left(u,\{ g:J^2(g\mu/R)\leq 1;||g\mu/R||^2\leq 1\} , \| \cdot \|_{\infty}\right) \\
&=&{\cal H}\left( \frac{u\mu}{R},g:J^2(g)\leq 1;||g||^2\leq 1, \| \cdot \|_{\infty}\right) \\
&\leq& A \biggl ({ R  \over u \mu  } \biggr ) ^{1 \over m} \quad,  \quad u >0 .
\ees

By Dudley's inequality (see e.g. \cite{vanderVaart:96}) we get for a universal constant $\CC$
\be
&&\EE \sup_{g: \ X\beta + g  \in {\cal F}(R) }   \biggl | {1 \over n} \sum_{i=1}^n g_i \varepsilon_i  \biggr |
\nonumber \\
&
{\leq}& \EE\frac{\CC}{\sqrt{n}} \int_0^{R_n} \left( {\cal H} ( u, \{  g 
: \ X \beta + g \in {\cal F} (R)\}  , \| 
\cdot \|_{\infty} ) \right)^{\frac{1}{2}} du \nonumber \\
&\leq& \EE\frac{\CC}{\sqrt{n}} \int_0^{R_n} \sqrt{A} \left(\frac{R}{u\mu}\right )^{\frac{1}{2m}}du \nonumber \\
&=&\EE \CC\sqrt{A}\frac{2m}{2m-1} \left(\frac{R_n^{\frac{2m-1}{2m}}}{\sqrt{n}}\right)\left(\frac{R}{\mu}
\right)^{\frac{1}{2m}} \label{KappaUnoGCalcolo.equation}\\
&=&\EE G^{-1} ( R_n^2 )/ \sqrt n \nonumber
\ee
Where 
$$
G^{-1} ( R_n^2):= \CC\sqrt{A}\frac{2m}{2m-1} R_n^{\frac{2m-1}{2m}}[R/\mu]^{\frac{1}{2m}}.
$$
One can now
easily see that $G^{-1}$ is the inverse of $G$ defined in \eqref{DefG.equation}.\\
Let now $\mathcal{G}(\tilde R,\tilde K):=\{g:X\beta+g\in\mathcal{F}(R), \|g\|\leq\tilde R,
\|g\|_\infty\leq\tilde K\}$.\\
Note that by Condition \ref{penalty.condition}, for $X\beta + g \in {\cal F} (R) $, 
$$ \| g \|_{\infty} \le \KG R / \mu ~.$$
Take now $\tilde R:=\Lambda R$ and $\tilde K:=\KG R/\mu$, then the conditions of 
Lemma \ref{classG.lemma} are fulfilled. After some easy computations we furthermore get


\be \label{def_H(u)_casoConcreto.eqation}
H(v)=\kappaUnoH \big(v\big)^{\frac{4m}{2m+1}}[R/\mu]^{\frac{2}{2m+1}}
\ee
where
\be \label{kappaUnoH}
\kappaUnoH :=\left[ \CC \sqrt{A}\frac{2m}{2m-1}\right]^{\frac{4m}{2m+1}}\cdot
\frac{(2m+1)(2m-1)^{\frac{2m-1}{2m+1}}}{(4m)^{\frac{4m}{2m+1}}}
\ee
and
$$
H\left( \frac{16\tilde K}{\sqrt{n}}\right)=\left[8\CC \sqrt{A}\KG\right]^{\frac{4m}{2m+1}} 
\frac{2m+1}{2m-1}n^{-\frac{2m}{2m-1}} [R/\mu]^2
$$
\be \label{kappaUno1.equation}
&\leq& \left[8\CC \sqrt{A}\KG\right]^{\frac{4m}{2m+1}} \frac{2m+1}{2m-1} \deltaUno^{\frac{4m}{2m+1}} R^2\\
&\leq& 2\Lambda^2 R^2 \nonumber
\ee


Consequently from Lemma \ref{classG.lemma} we have
\bes
\EE G^{-1} ( R_n^2 )/ \sqrt n &\leq&G^{-1} \left( 4 \Lambda^2 R^2 \right)/ \sqrt n \nonumber \\
&=&\CC\sqrt{A}\frac{2m}{2m-1} \left( 2\Lambda \right)^{\frac{2m-1}{2m+1}}  { R^{1 - {1 \over 2m}}  \over 
\sqrt n }  [ { R / \mu }]^{1 \over 2m} .
\ees
Resuming we showed
\be
\EE \sup_{g: \ X\beta + g  \in {\cal F}(R) }   \biggl | {1 \over n} \sum_{i=1}^n g_i \varepsilon_i  \biggr 
|\leq 
\CC\sqrt{A}\frac{2m}{2m-1} \left( 2\Lambda \right)^{\frac{2m-1}{2m+1}}  { R^{1 - {1 \over 2m}}  \over 
\sqrt n }  [ { R / \mu }]^{1 \over 2m} \label{kappaUno2.equation}
\ee
$$
\leq \kappaUno { R^{1 - {1 \over 2m}}  \over \sqrt n }  [ { R / \mu }]^{1 \over 2m}. 
$$
As in the proof of Lemma \ref{classG.lemma} and Lemma \ref{classG.corollary}
$$
\EE  \sup_{g: \ X\beta + g  \in {\cal F}(R) }  \biggl | \| g \|_n^2 - \| g \|^2 \biggr | \le 8 \tilde K \EE 
G^{-1} ( \tilde{R}_n^2 ) /\sqrt n .
$$
The right hand side of the inequality can be further approximated by
\bes
8K \EE G^{-1} ( \tilde{R}_n^2 ) /\sqrt n &\leq& \frac{8\KG R}{\mu\sqrt{n}}\EE G^{-1} ( \tilde{R}_n^2 )\\
&\leq&\frac{8\KG R}{\mu\sqrt{n}}\EE G^{-1} \left( 4\Lambda R^2 \right)\\
&\leq& \frac{8\KG R}{\mu}\CC\sqrt{A}\frac{2m}{2m-1} \left( 2\Lambda \right)^{\frac{2m-1}{2m+1}}  { R^{1 
- {1 \over 2m}}  \over \sqrt n }  [ { R / \mu }]^{1 \over 2m}\\
&\leq& 8\KG \kappaUno R^2 \mu^{-\frac{2m+1}{2m}} n^{-\frac{1}{2}}\\
&\leq& 8\KG \kappaUno \deltaUno R^2 .
\ees
This ends the proof.
\end{proof}

In the following lemma we finally show that the set $\mathcal{T}_1$ has large probability.
First we present a remark that will be useful for the proof.

\begin{remark} \label{ERn=cER.remark}
It holds that
\bes
\EE R_n \leq \big(\EE (R_n^2) \big)^{\frac{1}{2}} \leq \sqrt{ \EE  \sup_{g: \ X\beta + g  \in {\cal F}(R) }  
\biggl | \| g \|_n^2 - \| g \|^2 \biggr |+\|g\|^2} \\
\leq \sqrt{8\KG \kappaUno\deltaUno R^2 + \Lambda^2 R^2} = R\sqrt{8\KG \kappaUno\deltaUno + \Lambda^2 }
\ees
with $\Lambda= 1+ \Lambda_{\tilde X, {\rm min}} / \Lambda_{h, {\rm max}}$.
Here we used Jensen inequality for the first inequality and the proof of Lemma \ref{empiricalG.lemma} for the 
second one step.
\end{remark}

  \begin{lemma} \label{T1.lemma}
  
    Assume Conditions 
  \ref{design.condition}, \ref{eigenvalue.condition}, \ref{entropy.condition} and \ref{penalty.condition}, that
  $  \mu\le R $, \eqref{Assumption_EV_Lambda.equation} and that $R^2 \le  \lambda \le 1$.\\
  Then for constants $\deltaUnoP$ and $\kappaUnoP$ depending only
  on $\KX$,  $\KG$, $A$, $m$, $\Lambda$ (defined in \eqref{Assumption_EV_Lambda.equation})
, for $\deltaUnoP \lambda \ge \lambda_0   $, with $\lambda_0 = \sqrt {2\log (2p) / n} $,
    and
  $$\mu^2 \ge \kappaUnoP  n^{-{2m \over 2m+1} } , $$
we have
$$
\sup_{f \in {\cal F} (  R  ) } \biggl |\| f \|_n^2 - \| f \|^2   \biggr | \le \delta_0 R^2
$$
 with probability at least $1- \exp[-n (\deltaUnoP)^2 \mu^2  ]$.
 \end{lemma}

A suitable choice for $\kappaUnoP$ is
\be \label{kappaUnoP.equation}
\kappaUnoP:=\frac{\CC \sqrt{8\KG \kappaUno\deltaUno + \Lambda^2 }}{\sqrt{2}} + \kappaUno \KX^{\frac{1}{2m}} 
\sqrt{\delta_0/2}, 
\ee
where $\Lambda= 1+ \Lambda_{\tilde X, {\rm min}} / \Lambda_{h, {\rm max}}$, 
while  $\deltaUnoP$ can be chosen along the lines of Inequality \eqref{TuttiDelta.equation} or
$$
\deltaUnoP:=\frac{1}{3\delta_0}\left( 8 \delta_0 \KX^2+ 8\kappaUnoP \sqrt {\delta_0/2} \vee 
16 \kappaUno \KG \vee\kappaTre /\sqrt{2}\right) . 
$$

\begin{proof}
Let $\{ \varepsilon_i \}_{i=1}^n $ be a Rademacher sequence independent of $\{ ( x_i , z_i ) \}_{i=1}^n $. 
By symmetrization (see e.g. \cite{vanderVaart:96}), we have
$$\EE \sup_{f \in {\cal F} ( R) } \biggl | \| f \|_n^2 - \| f \|^2 \biggr | \le
2 \EE \biggl ( \sup_{f \in {\cal F} ( R)  } \biggl | 
{1 \over n} \sum_{i=1}^n f_i^2 \varepsilon_i \biggr | \biggr ) . $$

Note now that for $X \beta + g \in {\cal F} (R) $,
by Condition \ref{design.condition}
$$ \| X \beta \|_{\infty} \le \KX \| \beta \|_1 \le  \KX {R^2 \over \sqrt {2/\delta_0} \lambda } \le \KX 
\sqrt{\frac{\delta_0 }{2}}, $$ 
where we used the assumption that $R^2 / \lambda \le 1 $. 
By the contraction inequality (see \cite{Ledoux:91}), we have
$$\EE  \sup_{\beta: \ X \beta + g  \in {\cal F} ( R) } \biggl | {1 \over n} \sum_{i=1}^n ( X \beta )_i^2
\varepsilon_i \biggr | \leq
{4 \KX \over \sqrt {2/\delta_0 } }~ \EE  \sup_{\beta: \ X \beta + g  \in {\cal F} ( R) } \biggl | {1 \over n}
\sum_{i=1}^n ( X \beta )_i \varepsilon_i \biggr | .$$
But
\bes
\EE \sup_{\beta: \ X \beta + g  \in {\cal F} ( R) } \biggl | {1 \over n} \sum_{i=1}^n ( X \beta )_i
\varepsilon_i \biggr | &\leq& \EE \sup_{\beta: \ X \beta + g  \in {\cal F} ( R) }
\frac{1}{n}||\varepsilon^TX||_\infty ||\beta||_1 \\
&\leq& {R^2  \over \sqrt {2/ \delta_0} \lambda} \EE \| \varepsilon^T X \|_{\infty} /n  \\
&\leq&  \frac{ \lambda_0}{\lambda}  \KX \sqrt {\delta_0/2}   R^2\\
&\leq& \biggl ( \deltaUnoP \KX \sqrt {\delta_0/2}  \biggr ) R^2  ,
\ees
where we invoked the Cauchy-Schwarz inequality for the first step, Lemma 14.14 in \cite{BvdG2011} for the
second, the approximation $\EE \| \varepsilon^T X \|_{\infty} /n \le \lambda_0 \KX$ in the third one and the
assumption $\lambda \ge \lambda_0 /\deltaUnoP $ in the last inequality. 
Thus
$$\EE  \sup_{\beta: \ X \beta + g  \in {\cal F} ( R) } \biggl | {1 \over n} \sum_{i=1}^n ( X \beta )_i^2
\varepsilon_i \biggr | \le
4 \biggl (  { \deltaUnoP \KX^2 \over{2/\delta_0}}  \biggr ) R^2 = 2 \delta_0 \KX^2 \deltaUnoP R^2 .  $$

As in the proof of Lemma \ref{empiricalG.lemma}, we note that for $X \beta + g \in {\cal F} (R)$
$$ \left \| { g \over R/\mu } \right \|^2 \leq \Lambda^2\mu^2 \leq 1  $$
and
$$J \left ( { \mu g \over R} \right )  \le 1 . $$
Therefore, by Condition \ref{penalty.condition},
$$\| g \|_{\infty} \le  \KG R/\mu . $$
Again by the contraction inequality,
$$
 \EE \sup_{g: \ X \beta + g  \in {\cal F} ( R) } \biggl | {1 \over n} \sum_{i=1}^n g^2_i \varepsilon_i \biggr |
\le
{ 4 \KG  R \over \mu} \EE \sup_{g: \ X \beta + g  \in {\cal F} ( R) } \biggl | {1 \over n} \sum_{i=1}^n g_i
\varepsilon_i \biggr | . $$
From Lemma \ref{empiricalG.lemma}, for $\mu^2 \ge \kappaUno n^{-{2m \over 2m+1}}$, 
$$ \EE\sup_{g: \ X \beta + g  \in {\cal F} ( R) } \biggl | {1 \over n} \sum_{i=1}^n g_i \varepsilon_i \biggr |
\le 
\kappaUno { R  \over \sqrt n \mu^{1 \over 2m}  }   . $$
It follows that
$$ \EE \sup_{g: \ X \beta + g  \in {\cal F} ( R) } \biggl | {1 \over n} \sum_{i=1}^n g^2_i \varepsilon_i \biggr
| \le
\kappaUno { 4 \KG  R \over \mu}  { R  \over \sqrt n \mu^{1 \over 2m}  }$$
$$ = 4 \kappaUno \KG { R^2 \over \sqrt n \mu^{2m+1 \over 2m } } . $$

We also have
$$
\EE \sup_{\beta , g: \ X \beta + g  \in {\cal F} ( R) } \biggl | {1 \over n} \sum_{i=1}^n (X \beta)_i g_i
\varepsilon_i \biggr | \leq \EE \sup_{\beta , g: \ X \beta + g  \in {\cal F} ( R) }
||\beta||_1 \left\| \sum_{i=1}^{n}x_ig_i\epsilon_i \right\|_\infty 
$$
$$
\leq \frac{R^2}{\lambda} \sqrt {\delta_0/2 } \ \EE\max_{1 \le j \le p }  \sup_{   g: \ X \beta + g  \in {\cal F}
( R)
} \biggl | {1 \over n} \sum_{i=1}^n X_{ij} g_i \varepsilon_i \biggr | .
$$
Similarly as in the proof of Lemma \ref{empiricalG.lemma} we have
$$
\mathcal{H}(u,\{gX_j,g\in \mathcal{G},j=1,\ldots,p\},||.||_\infty)\leq \log \sum_{j=1}^p
N(u,gX_j,g\in \mathcal{G}\},||.||_\infty)
$$
$$
\leq \log p + \max_{j=1,...,p} \log N(u,gX_j,g\in \mathcal{G}\},||.||_\infty)
$$
$$
=\log p + \max_{j=1,...,p} \mathcal{H}(u,gX_j,g\in \mathcal{G}\},||.||_\infty)
\leq \log p + A \biggl (\frac{R\KX}{u\mu}\biggr )^{\frac{1}{m}},
$$
where in the last step we used that $||gX_j\mu/(R\KX)||\leq1$ and conditions \ref{design.condition} and
\ref{entropy.condition}.\\
In view of previous results, with help of Lemma \ref{empiricalG.lemma} we have that for $\mu^2 \ge 
\kappaUno n^{-{2 m \over 2m+1 }}$, by Dudley's inequality
$$
\EE\max_{1 \le j \le p }  \sup_{  g: \ X \beta + g  \in {\cal F} ( R) } \biggl | {1 \over n} \sum_{i=1}^n
X_{ij} g_i \varepsilon_i \biggr | $$
$$
{\le} ~\EE \frac{\CC}{\sqrt{n}}\int_0^{R_n}\big(\mathcal{H}(u,\{gX_j,g\in
\mathcal{G},j=1,\ldots,p\},||.||_\infty)\big)^{\frac{1}{2}} du
$$
$$
\leq \frac{\CC}{\sqrt{n}} \EE \sqrt{\log p}R_n  + \frac{\KX^{\frac{1}{2m}}}{\sqrt{n}} \EE G^{-1} (R_n^2)
$$
$$
\leq 
\frac{\CC \sqrt{8\KG \kappaUno\deltaUno + \Lambda^2 }}{\sqrt{2}} \lambda_0 R+\KX^{\frac{1}{2m}} 
\kappaUno \frac{R}{\sqrt{n}\mu^{\frac{1}{2m}}}
$$
In the last steps we used Remark \ref{ERn=cER.remark}, Lemma \ref{empiricalG.lemma} and the definition of
$\lambda_0$.

Hence for, $\mu^2 \ge (\kappaUno \vee \kappaUnoP \vee (\deltaUnoP){\frac{-4m}{2m+1}} ) n^{-{2m \over
2m+1}}$, $\lambda\leq\lambda_0/\deltaUnoP$ and $\mu\leq R$
$$ \EE \sup_{\beta , g: \ X \beta + g  \in {\cal F} ( R) } \biggl | {1 \over n} \sum_{i=1}^n (X \beta)_i g_i
\varepsilon_i \biggr | 
$$
$$
\le {R^2 \over \sqrt {2/\delta_0 } \lambda } \left[  \frac{\CC \sqrt{8\KG 
\kappaUno\deltaUno  
+ \Lambda^2 }}{\sqrt{2}} \lambda_0 R+\KX^{\frac{1}{2m}} 
\kappaUno \frac{R}{\sqrt{n}\mu^{\frac{1}{2m}}}\right] 
$$
$$
\leq \sqrt{\frac{\delta_0}{2}}R^2 \deltaUnoP \left( \frac{\CC \sqrt{8\KG \kappaUno\deltaUno + \Lambda^2 
}}{\sqrt{2}} + \kappaUno \KX^{\frac{1}{2m}} \sqrt{\delta_0/2} \right) 
$$
$$
= \kappaUnoP \sqrt{\frac{\delta_0}{2}}R^2 \deltaUnoP
$$
Resuming we have:
$$
\EE \sup_{\beta , g: \ X \beta + g  \in {\cal F} ( R) } \biggl | {1 \over n} \sum_{i=1}^n (X \beta)_i g_i
\varepsilon_i \biggr | \leq 2\kappaUnoP R^2  \sqrt {\delta_0 /2}\deltaUnoP~.
$$
Then, by symmetrization (Corollary 3.4 in \cite{vandeGeer:00}, \cite{vanderVaart:96}),
$$ \EE \sup_{f \in {\cal F} (R) } \biggl | \| f \|_n^2 - \| f \|^2 \biggr | \leq
2~\EE \sup_{f\in\mathcal{F}(R)}\frac{1}{n}\left|\sum_{i=1}^n f_i^2\varepsilon_i\right|$$
$$
=2~\EE \sup_{f\in\mathcal{F}(R)} \left|\sum_{i=1}^n \left( (X_\beta)_i^2\varepsilon_i +2 (X_\beta)_i
g_i\varepsilon_i+ g_i^2\varepsilon_i \right) \right|
$$
\be \label{Def_Expect_for_Massart.equation}
\leq 4 \delta_0 \KX^2 \deltaUnoP R^2  + 8 \kappaUno \KG \deltaUno { R^2  } + 8\kappaUnoP \sqrt {\delta_0/2} 
\deltaUnoP R^2 :=
{\bf E} .
\ee

We moreover have for any $f= X \beta + g $ in ${\cal F} (R) $, and for $\kappaTre:=\KX\delta_0/2+\KG$, 
$\| f \|_{\infty}\leq \|X\beta\|_\infty+\|g\|_\infty\leq \KX \|\beta\|_1+\|g\|_\infty \le \kappaTre R/ \mu$ and 
$\| f \| \le  R   $. Hence, $\| f^2 \| \le \| f \|_{\infty} \| f \| \le \kappaTre  R^2 / \mu$.

Therefore, for all $t>0$, by Massart's inequality \cite{Massart:00a}, (see also Corollary 14.2 in
\cite{BvdG2011}), using the approximation \\$\kappaTre R^2/\mu (2t)^{1/2}\cdot (1+2{\bf E}/R^2)^{1/2}\leq
\textbf{E} +
t\kappaTre^2R^2/\mu^2 + (2t)^{1/2} R^2 \kappaTre /\mu$,
$$ \PP \biggl ( \sup_{f \in {\cal F} (R)   } \biggl |\| f \|_n^2 - \| f \|^2   \biggr |  \ge 
2  {\bf E}    +
\kappaTre { R^2\over \mu}   \sqrt {2t \over  n} +  \frac{4}{3}  \kappaTre^2 { R^2 \over \mu^2 }   { t \over n} 
\biggr )$$ $$ \le \exp[-t] . $$
We now take
$$t:= n (\deltaUnoS)^2 \mu^2 $$
to find
$$ \PP \biggl ( \sup_{f \in {\cal F} (R)   } \biggl |\| f \|_n^2 - \| f \|^2   \biggr |  \ge 
2  {\bf E}    +
  \deltaUnoS \kappaTre R^2 \sqrt {2} + \frac{4}{3}\kappaTre^2  (\deltaUnoS)^2 R^2  \biggr )$$ $$ \le \exp\biggl 
[-n (\deltaUnoS)^2
\mu^2 
  \biggr ] . $$
  
Recall that \textbf{E} is defined in \eqref{Def_Expect_for_Massart.equation}. Take now $\deltaUno$, 
$\deltaUnoP$, $\deltaUnoS$ small enough, that
\be \label{TuttiDelta.equation} 
  8 \delta_0 \KX^2 \deltaUnoP   + 16 \kappaUno \KG \deltaUno  + 
16\kappaUnoP \sqrt {\delta_0/2} \deltaUnoP +
  \deltaUnoS \kappaTre   \sqrt {8} +  \frac{8}{3}  \kappaTre^2  (\deltaUnoS) ^2  \le \delta_0    .
\ee
So we have:
$$
\PP \biggl ( \sup_{f \in {\cal F} (R)   } \biggl |\| f \|_n^2 - \| f \|^2   \biggr |  \geq
\delta_0 R^2\biggr )\leq \exp\biggl [-n (\deltaUnoS)^2 \mu^2   \biggr ] .
$$

A possible choice for $\deltaUno$, $\deltaUnoP$, $\deltaUnoS$ respecting \eqref{deltaUno.equation}, is 
\be \label{deltaUnoP}
\deltaUnoP:=\deltaUnoS:=\frac{\delta_0}{3}\left( (8 \delta_0 \KX^2+ 16\kappaUnoP \sqrt {\delta_0/2} \vee 
16 \kappaUno \KG )\vee(4.6 \kappaTre )\right)^{-1}~,
\ee
and
\be \label{deltaUnoSecondaDef}
\deltaUno:=\min\{\deltaUnoP,\eqref{deltaUno.equation}\}~.
\ee
\end{proof}

\subsection{The set ${\cal T}_2( \delta_0 , R )$}

\begin{lemma}\label{T2.lemma} 
   Assume Conditions \ref{Gaussian.condition}, 
  \ref{design.condition}, \ref{eigenvalue.condition}, \ref{entropy.condition} and \ref{penalty.condition},
  that $ \mu^2 \le R^2$, \eqref{Assumption_EV_Lambda.equation} and $R^2 \le  \lambda \le 1$.\\
  Then for constant $\deltaDUE$ depending only on $\KX$,  $\KG$, $A$, $m$, for $\deltaDUE \lambda \ge \lambda_0 
R/\mu $, $\lambda_0 = \sqrt {2\log (2p )/ n} $ and $\mu^2 \ge \kappaUno \vee \kappaUnoP n^{-{2m \over 2m+1} }$, 
we have
$$
\sup_{f \in {\cal F} (  R  ) } \biggl | E^T f / n    \biggr | \le \delta_0 R^2~,
$$
with probability at least $1- 2 \exp[-n \deltaDUE^2 \mu^2  ]$
 \end{lemma}

{\bf Proof.} 
For $X$ fixed we have 
$$E^TX_j\sim\mathcal{N}(0,||X_j||_2^2)$$
and consequently
\be \label{Zwischenstep_T2.equation}
\EE\max_{1\leq j\leq p} E^TX_j\leq \max_{1\leq j\leq p} \left( \sum_{i=1}^n X_{ij}^2\right)
^{\frac{1}{2}} \sqrt{2\log p}\leq \KX \sqrt{2\log p}\sqrt{n}
\ee
We continue as in the proof of Lemma \ref{T1.lemma}, but now with the Rademacher sequence replaced by the
Gaussian errors $\{ e_i \}$, and conditionally on $X$.
Then we obtain
\bes
\EE\left (  \sup_{\beta:\ X \beta + g \in {\cal F} (R) } \biggl | {1 \over n} \sum_{i=1}^n (X\beta)_i
e_i \biggr | \biggr \vert X 
\right ) \le {R^2  \over \sqrt {2/ \delta_0} \lambda}  
\EE \biggl ( \| E^T X \|_{\infty} \vert X \biggr ) /n
\ees
\bes
 \le  { \lambda_0  \KX \over \sqrt{2/\delta_0}  \lambda} R^2 \le{\deltaDUE \KX \over\sqrt {2/\delta_0 }}{R\mu} 
\le
 { \deltaDUE \KX \over  \sqrt {2/\delta_0 } } {R^2},
\ees
where we used \eqref{Zwischenstep_T2.equation}, the definition of $\lambda_0$ and Assumption
\eqref{Assumption_lambda0.equation}. Moreover, conditionally on $Z$
$$
\EE \left ( \sup_{g:\ X\beta + g \in {\cal F} (R)   } \biggl | { 1 \over n}  \sum_{i=1}^n g_i e_i \biggr
| \biggr \vert Z \right ) =\EE \left ( \sup_{g:\ X\beta + g \in {\cal F} (R)   } \biggl | { 1 \over n} 
\sum_{i=1}^n g_i e_i \epsilon_i \biggr | \biggr \vert Z \right ) 
$$
Notice that $\|g e\|= \|g\|\cdot \|e\|=\|g\|\leq \Lambda R$, where
$\Lambda := 1 + \Lambda_{h, {\rm max} } / \Lambda_{\tilde X , {\rm min}}$.
Then the above inequality can be further 
approximated,
with help of Lemma \ref{empiricalG.lemma} as follows:
$$
\leq \kappaUno { R_n^{1- {1 \over 2m} }   \over \mu^{2m+1 \over 2m} \sqrt n   } R^{1 \over 2m} \mu .
$$
Recall that $R_n = \sup_{g:\ X\beta + g \in {\cal F} (R)   } \| g \|_n$.
For $\mu \ge \deltaDUE^{-{2 m \over 2m+1}} n^{-{m \over 2m+1} } $, we obtain
$$
\EE \left( \sup_{f \in {\cal F} (R)} \frac{|E^T f |}{n}\right) \leq \EE \left( \sup_{g:\ X\beta + g \in {\cal F}
(R)} \frac{|E^T g |}{n} \right)  +\EE \left ( \sup_{\beta:\ X\beta + g \in {\cal F} (R)} \frac{|E^T X\beta 
|}{n}\right)
$$
$$
\leq \kappaUno { R_n^{1- {1 \over 2m} }   \over \mu^{2m+1 \over 2m} \sqrt n   } R^{1
\over 2m} \mu + \deltaDUE \KX \sqrt{\delta_0/2} R^2
$$
Assuming $R_n\leq 2R$ we would have
$$
\EE \left( \sup_{f \in {\cal F} (R)} \frac{|E^T f |}{n}\right) \leq 2\kappaUno R^2\deltaDUE + \deltaDUE \KX
\sqrt{\delta_0/2} R^2 .
$$
Define now the event $\mathbb{A}$ as follows:
$$
\mathbb{A}:=\left\lbrace \sup_{f \in {\cal F} (R)} | \epsilon^T f /n | \ge  {  \deltaDUE \KX \over  \sqrt
{2/\delta_0 } } {
R^2  } +  2 \kappaUno \deltaDUE R^2 + 2 R \sqrt {\frac{8t}{n}} \right\rbrace 
$$
and 
$$
\tilde{R}_n= \sup_{f\in \mathcal{F}(R)} ||f||_n
$$
Then, by concentration of measure for Gaussian random variables (see \cite{massart1896concentration}), for all
$t>0$,
$$
\PP (\mathbb{A})=\EE\left(\PP \left( \mathbb{A} \Big|Z\right)  \right)=
\EE\left(\PP \left( \mathbb{A} \wedge \tilde{R}_n\leq 2R\Big|Z\right)  \right) + \EE\left(\PP \left( \mathbb{A}
\wedge \tilde{R}_n> 2R\Big|Z\right) \right)
$$
$$
\leq \EE\left(\PP \left( \mathbb{A} \wedge \tilde{R}_n\leq 2R\Big|Z\right)  \right)+ \EE\left(\PP
\left(\tilde{R}_n> 2R\Big|Z\right) \right) \leq e^{-t}+\PP \left(\tilde{R}_n>2R\right)
$$
$$
\leq e^{-t}+\PP \left(\tilde{R}_n^2>4R^2\right) \leq e^{-t}+\PP \left(\tilde{R}_n^2- \sup_{f\in
\mathcal{F}(R)}||f||^2 \geq 3R^2 \right) 
$$
$$
\leq e^{-t}+\PP \left( \sup_{f \in \mathcal{F}(R)} \Big|
||f||_n^2-||f||^2\Big|\geq 3R^2 \right) 
$$
$$
\leq e^{-t} +\PP \left(\sup_{f \in \mathcal{F}(R) }\Big|
||f||_n^2-||f||^2\Big|\geq \delta_0 R^2 \right)
$$
$$
\leq  e^{-t} +e^{-\deltaDUE^2 \mu^2 n}
$$
where in the last step we used Lemma \ref{T1.lemma}.
Choose now:
$$
t:=\deltaDUE^2 \mu^2 n
$$
small enough such that
$$
{\deltaDUE \KX \over  \sqrt {2/\delta_0 } } { R^2  } +  2 \kappaUno \deltaDUE R^2 + 2 R \sqrt {\frac{8t}{n}} 
\leq
\delta_0 R^2 .
$$
For example:
\be \label{deltaDUE}
\deltaDUE:=\delta_0\left( \KX  \sqrt {\delta_0 /2} + 2 \kappaUno + 4 \sqrt{\delta_0}\right)^{-1} ~ .
\ee
Then we have
$$
\PP\left( {\cal T}_2( \delta_0 , R ) \right) \geq 1-2 e^{-n\deltaDUE^2 \mu^2}.
$$
\hfill $\sqcup \mkern -12mu \sqcap$

\begin{proof} \textit{of Theorem \ref{1.theorem}}\\
We first assume that we are on $\mathcal{T}$, then as direct consequence of Lemma \ref{Teo_T.lemma}, all the 
conclusions of Theorem \ref{1.theorem} holds. It is now enough to show that the probability of $\mathcal{T}$ is 
at least as large as required for satisfying the theorem's requests.\\
Along the lines of \eqref{deltaUnoSecondaDef}, \eqref{deltaDUE}, and 
\eqref{kappaUno.equation}, \eqref{kappaUnoP.equation}, we can now define
\be \label{deltaTOT.equation}
\deltaTOT:=\min\{ \deltaUno,\deltaDUE\}
\ee
and
\be \label{kappaTOT.equation}
\kappaTOT:=\max\{\kappaUno,\kappaUnoP\}
\ee
Then we have
$$
\PP [\mathcal{T}]\geq 1-\PP[\mathcal{T}_1^c]-\PP[\mathcal{T}_2^c]\geq 1-\exp\biggl [-n (\deltaUnoS)^2 \mu^2   
\biggr ]-2\exp\biggl [-n (\deltaDUE)^2 \mu^2 \biggr ]
$$
$$
\leq 1-3 \exp \biggl [-n (\deltaTOT)^2 \mu^2 \biggr ] .
$$
\end{proof}

\begin{proof} \textit{of Theorem \ref{2.theorem}}\\
The idea of the proof is that $(\hat\beta,\hat g)$ is obtained by minimising the penalised sum of squares
so the ``derivative" has to be $0$ in $(\hat\beta,\hat g)$. This is the so called 
Karash-Kuhn-Tucker- (KKT-) condition. Similar 
ideas as in the proof of Theorem \ref{1.theorem} are used in Lemmas 
\ref{helpLemma1.lemma}-\ref{gX_unif_bound.lemma}, which provide useful approximations for finishing the 
proof.
\bigskip \\
Define
$$
\hat \beta_s^j:= \hat\beta+s {\rm e}_j,~~\hat g_s^j := \hat g -s h^j,
$$
where ${\rm e}_j$ is the $j-$th unit vector of the standard basis of $\R^p$. Let then $\hat \tau$ satisfy
$\|\hat \tau \|_\infty\leq 1$ and $\hat \tau^T \hat \beta=\|\hat\beta\|_1$. The KKT-condition is given by
$$
\frac{d}{ds} \left( \|Y-X\hat\beta^j_s-\hat g^j_s\|_n^2 +\lambda \|\hat\beta^j_s\|_1+\mu^2 J^2(\hat
g^j_s)\right)\bigg|_{s=0}=0,
$$
where the ``derivative" is to be understood in sense of subdifferential calculus. Differentiating and using the 
matrix notation we get
$$
-2(Y-X\hat\beta-\hat g)^T \tilde X/n+\lambda \hat \tau-2 \mu^2 J(\hat g,h)=0,
$$
Writing $Y=X\beta^0 +g^0+E$ we have
$$
2\Big(X(\hat\beta-\beta^0)+(\hat g-g^0)-E \Big)^T \tilde X/n+\lambda \hat \tau-2 \mu^2 J(\hat g,h)=0
$$
Hence
$$
2(\hat\beta-\beta^0)^T\tilde X^T \tilde X/n+\lambda \hat \tau + 2(\hat\beta-\beta^0)^T h^T \tilde X/n
$$
$$+2(\hat g-g^0)^T \tilde X/n+2E^T\tilde X/n -2 \mu^2 J(\hat g,h)=0.
$$
Multiplying by $\hat\beta-\beta^0$ we obtain
$$
2(\hat\beta-\beta^0)^T\tilde X^T \tilde X (\hat\beta-\beta^0)/n +\lambda \|\hat\beta\|
-\hat\tau^T\beta^0+2(\hat\beta-\beta^0)^T h^T \tilde X(\hat\beta-\beta^0)/n
$$ $$
+2(\hat g-g^0)^T \tilde X(\hat\beta-\beta^0)/n +2E^T\tilde X(\hat\beta-\beta^0)/n-2 \mu^2 J(\hat
g,h)(\hat\beta-\beta^0)=0 .
$$
Note that $\hat\tau \beta^0\leq \|\beta^0\|_1$ and hence $\hat\tau^T\beta^0-\|\hat\beta_{S_0}\|_1\leq
\|(\hat\beta-\beta^0)_{S_0}\|_1$. Therefore, we have
$$
2(\hat\beta-\beta^0)^T\tilde X^T \tilde X (\hat\beta-\beta^0)/n+\lambda \|\hat\beta_{S_0^c}\|_1
$$
$$
\leq \lambda \|(\hat\beta-\beta^0)_{S_0}\|_1 + 2|(\hat\beta-\beta^0)^T h^T \tilde X(\hat\beta-\beta^0)/n| 
$$ $$
+2|(\hat g-g^0)^T \tilde X(\hat\beta-\beta^0)/n|+2|E^T\tilde X(\hat\beta-\beta^0)/n| + 2|\mu^2 J(\hat
g,h)(\hat\beta-\beta^0)| .
$$
We approximate separately each term of the right hand side of the inequality. We use Lemmas
\ref{helpLemma1.lemma}-\ref{gX_unif_bound.lemma} and similar arguments as in the proof of Lemmas
\ref{T1.lemma}-\ref{T2.lemma} for the approximation:
\bes
&\leq& \lambda \|(\hat\beta-\beta^0)_{S_0}\|_1+ 2\biggl( 4 \kappaUno \deltaUno   + \KX  \sqrt {2} 
\deltaTre^2 + { 
8 \KX \deltaTre^2   \over 3 }\biggr ) R\mu \| \hat \beta - \beta^0 \|_1  \\
&&+\left( 7  \sqrt {2\delta_0} \deltaTOT R^2  + 2\sqrt {2 \delta_0} R^2 +4\sqrt{2}\KX 
\sqrt{\frac{2\log(2p)}{n}}\right) \| \hat \beta - \beta^0 \|_1 .
\ees
Recall that $R\mu \le R^2\leq \lambda$. Choose now $\delta_0, \deltaUno $ and $\deltaTre$ small enough that
$$
2\biggl( 4 \kappaUno \deltaUno   + \KX  \sqrt {2} \deltaTre  + { 8 \KX \deltaTre^2   \over 3 }\biggr ) R\mu  
+7  \sqrt {2\delta_0} \deltaTOT R^2 +
$$
$$
+ 2\sqrt {2 \delta_0} R^2 +4\sqrt{2}\KX \sqrt{\frac{2\log(2p)}{n}} \leq \frac{\lambda}{2}
$$
Hence
\be \label{Step_Thm2_Proof.equation}
2\|\tilde X (\hat\beta-\beta^0)\|_n^2 +\frac {\lambda}{2} \|\hat\beta_{S_0^c}\|_1\leq
\frac{3 \lambda}{2} \|(\hat\beta-\beta^0)_{S_0}\|_1 .
\ee
Adding $\lambda \|(\hat\beta-\beta^0)_{S_0}\|_1/2$ on both sides of
\eqref{Step_Thm2_Proof.equation} we obtain:
$$
2\|\tilde X (\hat\beta-\beta^0)\|_n^2 +\frac{\lambda}{2}
\|\hat\beta-\beta^0\|_1\leq 2\lambda \|(\hat\beta-\beta^0)_{S_0}\|_1
$$
$$ \le 2 \lambda \sqrt {s_0} \| \hat \beta - \beta_0 \|_2 $$
$$ \le 2 \lambda \sqrt {s_0} \| \tilde X ( \hat \beta - \beta_0 ) \|_2 / \Lambda_{\tilde X , {\rm min} }  $$
$$ \le {  \lambda^2 s_0  \over \Lambda_{\tilde X, {\rm min}}^2 } + \| \tilde X ( \hat \beta - \beta_0) \|^2 $$
$$ \le {  \lambda^2 s_0  \over \Lambda_{\tilde X, {\rm min}}^2 } + 
\| \tilde X ( \hat \beta - \beta_0) \|_n^2 + \lambda \| \hat \beta - \beta^0 \|_1 / 4  ,$$
where we apply the third part of Lemma \ref{gX_unif_bound.lemma},
with $20  \sqrt {\delta_0/2} \deltaTOT R^2  \le \lambda / 4 $. 
We find
$$ \|\tilde X (\hat\beta-\beta^0)\|_n^2 +\frac{\lambda}{4} \| \hat \beta - \beta_0 \|_1 \le 
{  \lambda^2 s_0  \over \Lambda_{\tilde X, {\rm min}}^2 } . $$

This with probability at least $1-3/(2p)- \exp[-n ( \deltaTre \mu^2) ] - \PP ({\cal T}^c ) $.
Lemmas \ref{T1.lemma} and \ref{T2.lemma} conclude the
proof.
\end{proof}
 
\begin{lemma}\label{helpLemma1.lemma}
Assume Conditions \ref{Gaussian.condition} and \ref{design.condition}. Then  
$$
E^T\tilde X(\hat\beta-\beta^0)/n\leq \sqrt{8} \KX \sqrt { 2 \log (2p) \over n} \|\hat\beta-\beta^0\|_1
$$
with probability at least $1-1/p$
\end{lemma}
\begin{proof} We have
$$
E^T\tilde X(\hat\beta-\beta^0)/n\leq \| E^T\tilde 
X/n\|_\infty \|\hat\beta-\beta^0\|_1.
$$
For $\tilde X$ fixed, $E^T \tilde X $ is Gaussian. 
So for all $t >0$ and all $j$ 
$$\PP ( | E^T\tilde 
X_j/n|  \ge \sqrt {\frac{2 t}{n}} \| X_j \|_n  ) \le 2 \exp[-t] . $$
Hence
$$\PP\biggl  ( \| E^T\tilde 
X/n\|_{\infty}   \ge \sqrt {\frac{2 (t+ \log (p ))}{n} } \max_{1 \le j \le p }  \| X_j \|_n \biggr  ) \le 2 
\exp[-t] . $$

Moreover, $\max_{1 \le j \le p} \| \tilde X_j \|_n \le \| \tilde X \|_{\infty} \le 2 \KX$. 
Now take $t = \log (2p)$.
\end{proof}

\begin{lemma}
Assume Conditions \ref{eigenvalue.condition}, \ref{Jh.condition}, \eqref{condition_mu_Thm.equation} and 
\eqref{Assumption_R2Lambda2}. Then on $\mathcal{T}$
$$
\mu^2 J(\hat g,h) (\hat\beta-\beta^0)\leq \sqrt{2 \delta_0} R^2 \|\hat \beta -\beta^0\|_1
$$
\end{lemma}
\begin{proof}
Similarly to the previous lemma we have
$$
\mu^2 J(\hat g,h) (\hat\beta-\beta^0)\leq \max_{j\in 1,...,p} \mu^2 J(\hat g,h_j) \|\hat\beta-\beta^0\|_1
$$
Hence, by triangle inequality and Lemma \ref{Teo_T.lemma}
$$
\max_{j\in 1,...,p}|J(\hat g,h_j)|\leq \Jh J(\hat g)\leq \frac{R}{\mu} \Jh+\Jh J(g_0)
$$
and
$$
\max_{j\in 1,...,p} \mu^2 |J(\hat g,h_j)| \leq \sqrt{ \delta_0/2} R^2
$$
\end{proof}

\begin{lemma} \label{gX_unif_bound.lemma}
Assume the conditions of Theorem \ref{1.theorem}. Then for all $\deltaTre >0$, with $\deltaTre\mu\geq 
\sqrt{\log 
p/n}$ except for a set of probability at
most $\exp[-n ( \deltaTre^2 \mu )^2] $, one has
$$  \biggl | {1 \over n} \sum_{i=1}^n g_i \tilde x_i \beta  \biggr | \le 
\biggl ( 4 \kappaUno \deltaUno   + \KX  \sqrt {2} \deltaTre^2  + { 8 \KX \deltaTre   \over 3 }\biggr ) 
R \mu \| \beta \|_1 $$
uniformly in $(\beta , g )  \in {\cal F} (R) $.
Furthermore, except for a set with probability at most $1/ (2p)$, 
and
$$  \biggl | {1 \over n} \sum_{i=1}^n (h \beta )_i (\tilde X\beta)_i  \biggr | \le 7  \sqrt {\delta_0/2} 
\deltaTOT R^2 \| \beta \|_1  . $$
uniformly in $(\beta , g)  \in {\cal F} (R) $.\\
Finally, also up to a set with probability at most $1/ (2p)$, 
$$ \biggl |  \| \tilde X \beta \|_n^2 - \| \tilde X \beta \|^2 \biggr | \le 20  \sqrt {\delta_0/2} \deltaTOT 
R^2  \| \beta \|_1  . $$
uniformly in $(\beta , g)  \in {\cal F} (R) $.

\end{lemma}

\begin{proof} As in Lemma \ref{empiricalG.lemma} (applying a symmetrisation step), since by assumption $\mu^2 
\geq 
\deltaUno^{-{ 4m \over 2m+1}} n^{-{42m \over 2m+1}} $, we have for all $j$, 
$$\EE \sup_{g: \ X \beta + g \in {\cal F} (R) } 
\biggl | {1 \over n} \sum_{i=1}^n g_i \tilde X_{ij} \biggr | \le 4 \kappaUno \deltaUno \KX R\mu .$$ 
By Massart's inequality, for all $t > 0$, and all $j$
$$\PP \biggl ( \sup_{ g:\ \ X \beta + g \in {\cal F} (R) } \biggl | {1 \over n} \sum_{i=1}^n g_i  \tilde X_{ij} 
\biggr | 
\ge 8 \kappaUno \deltaUno R \mu + \KX R \sqrt {2t / n} + { 4 \KX R  t \over 3\mu n } \biggr ) \le \exp[-t] . $$
It follows that for all $t >0$ 
$$\PP \biggl ( \sup_{ g, j, \ X \beta + g \in {\cal F} (R) } \biggl | {1 \over n} \sum_{i=1}^n g_i \tilde 
X_{ij} \biggr | $$ $$
\ge 8 \kappaUno \deltaUno R \mu + \KX R \sqrt {2(t + \log p )\over  n} + { 4 \KX R ( t + \log p ) \over 3\mu n 
} 
\biggr ) \le \exp[-t] . $$
Choose $t = n ( \deltaTre \mu )^2 $ and use that $\deltaTre \mu \ge \sqrt {\log p / n } $ to get
$$\PP \biggl ( \sup_{g,  j, \ X \beta + g \in {\cal F} (R) } \biggl | {1 \over n} \sum_{i=1}^n g_i \tilde 
X_{ij} \biggr | 
\ge 8 \kappaUno \deltaUno R \mu + \KX R \sqrt {2} \deltaTre \mu + { 8 \KX R\mu \deltaTre   \over 3 } \biggr ) 
$$ 
$$\le \exp[-n (\deltaTre \mu )^2 ] . $$
Now apply 
$$ \biggl | {1 \over n} \sum_{i=1}^n g_i \tilde x_i \beta  \biggr | 
\le  \max_{1 \le j \le p }  \biggl | {1 \over n} \sum_{i=1}^n g_i \tilde X_{ij} \biggr | \| \beta \|_1 . $$

For the second result, we use
$$\EE \max_{j,k} \biggl | {1 \over n} \sum_{i=1}^n h_{ik} \tilde X_{ij} \biggr | \le 
\sqrt 2 {4 \log (2p) \over n} \KX^2 , $$
by Lemma 14.14 in \cite{BvdG2011}. Using Massart's inequality (Theorem 14.2 in \cite{BvdG2011}),
we get for all $t > 0$, 
$$\PP \biggl ( \max_{j,k} \biggl | {1 \over n} \sum_{i=1}^n h_{ik} \tilde X_{ij} \biggr | \ge
2 \sqrt {4 \log (2p) \over n} \KX^2 + 2 \KX^2 \sqrt {8 t \over n }  \biggr ) \le \exp [-t] .$$
Choosing $t = \log (2p)$ gives
$$\PP \biggl ( \max_{j,k} \biggl | {1 \over n} \sum_{i=1}^n h_{ik} \tilde X_{ij} \biggr | \ge
7 \sqrt {2 \log (2p) \over n} \KX^2   \biggr ) \le {1 \over 2p}  .$$
Now apply that for $\beta$ such that $X\beta + g \in {\cal F} (R) $,
$$ \biggl | {1 \over n} \sum_{i=1}^n (h \beta )_i \tilde X_{ij} \biggr | \le \sqrt {\delta_0/2} {R^2 \over 
\lambda} \max_{1\leq k\leq p} 
\biggl | {1 \over n} \sum_{i=1}^n h_{ik} \tilde X_{ij} \biggr | , $$
and the bound 
$$\lambda \ge { \sqrt {2 \log (2p )/n } \over \deltaTOT } . $$

For the third result, we use the same arguments as for the second one. By Lemma 14.14 in \cite{BvdG2011}
$$\EE \max_{j,k}  \biggl | {1 \over n} \sum_{i=1}^n ( \tilde X_{ik} \tilde X_{ij} - \EE\tilde X_{ik} \tilde 
X_{ij}) \biggr | \le
4 \sqrt {4 \log (2p) / n } \KX^2 , $$
and by Massart's inequality for all $t$,
$$ \PP (  \max_{j,k}  \biggl | {1 \over n} \sum_{i=1}^n ( \tilde X_{ik} \tilde X_{ij} - \EE\tilde X_{ik} \tilde 
X_{ij}) \biggr |
\ge 4 \sqrt {4 \log (2p) / n } \KX^2  + 4 \KX^2 \sqrt {8t } \biggr ) \le \exp [-t] . $$
So for $t=\log (2p)$,
$$\PP (  \max_{j,k}  \biggl | {1 \over n} \sum_{i=1}^n ( \tilde X_{ik} \tilde X_{ij} - \EE\tilde X_{ik} \tilde 
X_{ij}) \biggr |
\ge 20 \sqrt {4 \log (2p) / n } \KX^2  \biggr ) \le 1/(2p)  . $$
Then again apply H\"older's inequality and $\| \beta \|_1 \le \sqrt {\delta_0/2} R^2/\lambda$ whenever
$(\beta , g ) \in {\cal F}(R)$.

\end{proof}

\clearpage

\section{Tables with results} \label{NumRes.Appendix}
Here the results of the simulations are resumed in tables \ref{ResultsLe1.table}-\ref{ResultsPr2.table}.
\begin{table} \footnotesize \renewcommand{\tabcolsep}{3.5pt}\centering 
\begin{tabular}{||r|l || c|c|c|c || c|c|c|c || c|c ||}
\hline\hline
\multicolumn{12}{||c||}{Results for the set Leukemia}\\
\hline\hline
$n^\circ$&setting & \multicolumn{4}{c||}{Estimation error}&\multicolumn{4}{c||}{Prediction
error}&\multicolumn{2}{c||}{$\|\hat g-g^0\|_n$}\\
\hline
$N$&$p$,$s_0$,STN,$g_j$&DPI&DPd&LK&LN &DPi&DPd&LK&LN &DPi&DPd\\
\hline \hline
1&250,5,2,$g_1$&0.87&0.91&0.71&0.71&4.81&5.83&3.95&3.95&0.60&0.95 \\ \hline
2&250,15,2,$g_1$&1.84&1.84&1.71&1.71&18.45&18.35&16.90&16.92&1.34&1.56 \\ \hline
3&1000,5,2,$g_1$&0.97&0.99&0.83&0.83&6.56&7.16&5.28&5.32&0.64&1.03 \\ \hline
4&1000,15,2,$g_1$&1.97&1.98&1.95&1.96&20.67&20.54&19.49&49.47&1.41&1.65 \\ \hline
5&250,5,8,$g_1$&0.21&0.23&0.17&0.17&1.13&1.72&0.94&0.95&0.15&0.32 \\ \hline
6&250,15,8,$g_1$&0.52&0.54&0.46&0.46&11.74&12.67&8.14&8.24&0.65&0.98 \\ \hline
7&1000,5,8,$g_1$&0.23&0.26&0.2&0.2&1.66&2.98&1.36&1.38&0.17&0.53 \\ \hline
8&1000,15,8,$g_1$&0.77&0.78&0.85&0.85&18.92&18.78&17.2&17.25&1.02&1.32 \\ \hline
9&250,5,32,$g_1$&0.06&0.07&0.05&0.05&0.23&0.40&0.19&0.19&0.04&0.10 \\ \hline
10&250,15,32,$g_1$&0.24&0.24&0.17&0.17&8.54&9.85&3.61&3.72&0.47&0.80 \\ \hline
11&1000,5,32,$g_1$&0.07&0.08&0.06&0.06&0.38&1.18&0.30&0.30&0.05&0.26 \\ \hline
12&1000,15,32,$g_1$&0.48&0.53&0.60&0.57&18.66&18.57&16.60&16.72&0.96&1.29 \\ \hline
13&250,5,2,$g_2$&0.86&0.89&0.71&11.74&4.69&5.50&3.89&6.06&0.59&0.96 \\ \hline
14&250,15,2,$g_2$&1.85&1.84&1.71&11.84&18.34&18.45&16.82&17.86&1.37&1.58 \\ \hline
15&1000,5,2,$g_2$&0.96&0.97&0.83&11.74&6.16&6.65&5.22&7.06&0.67&1.04 \\ \hline
16&1000,15,2,$g_2$&1.99&1.96&1.96&11.89&20.38&20.50&19.32&19.73&1.43&1.66 \\ \hline
17&250,5,8,$g_2$&0.21&0.22&0.17&11.72&0.98&1.52&0.96&5.28&0.15&0.32 \\ \hline
18&250,15,8,$g_2$&0.54&0.55&0.46&11.72&11.90&12.82&8.28&15.79&0.68&1.00 \\ \hline
19&1000,5,8,$g_2$&0.23&0.27&0.20&11.71&1.52&2.85&1.40&6.37&0.18&0.56 \\ \hline
20&1000,15,8,$g_2$&0.84&0.81&0.83&11.77&18.55&18.68&17.22&18.96&1.04&1.33 \\ \hline
21&250,5,32,$g_2$&0.17&0.17&0.05&11.72&0.54&0.84&0.19&5.21&0.09&0.25 \\ \hline
22&250,15,32,$g_3$&0.28&0.29&0.16&11.72&8.95&10.24&3.51&15.42&0.51&0.86 \\ \hline
23&1000,5,32,$g_2$&0.17&0.18&0.06&11.71&0.72&1.70&0.29&6.28&0.10&0.40 \\ \hline
24&1000,15,32,$g_2$&0.59&0.62&0.60&11.76&18.41&118.35&16.44&18.76&1.03&1.34 \\ \hline
25&250,5,2,$g_3$&0.87&0.90&0.71&3.97&4.79&5.62&3.94&7.05&0.60&0.94 \\ \hline
26&250,15,2,$g_3$&1.84&1.82&1.72&4.24&18.64&18.48&16.87&18.72&1.34&1.56 \\ \hline
27&1000,5,2,$g_3$&0.97&0.99&0.83&3.99&6.44&6.95&5.27&7.76&0.66&1.02 \\ \hline
28&1000,15,2,$g_3$&1.98&1.94&1.95&4.36&20.67&20.61&19.40&20.07&1.41&1.62 \\ \hline
29&250,5,8,$g_3$&0.21&0.23&0.17&3.89&1.12&1.68&0.94&6.56&0.15&0.32 \\ \hline
30&250,15,8,$g_3$&0.52&0.54&0.45&3.91&11.66&12.91&8.18&17.51&0.65&1.00 \\ \hline
31&1000,5,8,$g_3$&0.23&0.26&0.20&3.89&1.68&2.95&1.38&7.52&0.18&0.53 \\ \hline
32&1000,15,8,$g_3$&0.76&0.80&0.83&4.02&18.88&18.79&17.27&19.76&1.02&1.34 \\ \hline
33&250,5,32,$g_3$&0.08&0.08&0.05&3.88&0.26&0.48&0.19&6.62&0.05&0.13 \\ \hline
34&250,15,32,$g_3$&0.24&0.25&0.17&3.89&8.52&10.05&3.57&17.36&0.47&0.81 \\ \hline
35&1000,5,32,$g_3$&0.08&0.10&0.06&3.89&0.43&1.32&0.29&7.53&0.06&0.29 \\ \hline
36&1000,15,32,$g_3$&0.53&0.53&0.58&4.02&18.39&18.64&16.55&19.67&0.98&1.27 \\ \hline
\hline
\end{tabular}
\caption{The results of the pseudo real data study with the dataset Leukemia are presented here. For 
the 36 different designs ($N=1,\ldots,36$) the performance of Lasso with (LK) and without (LN) prior 
knowledge on the nuisance function are compared with our estimator in both the independent (DPi) and dependent 
(DPd) case. Prediction and estimation error for $\hat\beta$ and $\hat g$ are given. The results are based 
on 1000 replicates for each design.}
\label{ResultsLe1.table}
\end{table}

\begin{table} \footnotesize \centering 
\begin{tabular}{||r|l || c|c|c|c || c|c|c|c || }
\hline\hline
\multicolumn{10}{||c||}{Results for the set Leukemia}\\
\hline\hline
$n^\circ$&setting &\multicolumn{4}{c||}{TPR}&\multicolumn{4}{c||}{FPR}\\
\hline
$N$&$p$,$s_0$,STN,$g_i$&DPi&DPd&LK&LN &DPi&DPd&LK&LN\\
\hline \hline
1&250,5,2,$g_1$&97.9&84.0&98.8&98.8&11.4&12.0&9.8&9.7 \\ \hline
2&250,15,2,$g_1$&56.7&51.9&65.1&64.6&13.2&12.6&12.8&12.7 \\ \hline
3&1000,5,2,$g_1$&88.8&68.8&93.9&93.5&4.0&3.9&3.3&3.3 \\ \hline
4&1000,15,2,$g_1$&30.3&27.7&35.4&34.9&3.7&3.6&3.4&3.3 \\ \hline
5&250,5,8,$g_1$&99.9&99.3&100&100&10.4&13.3&9.5&9.3 \\ \hline
6&250,15,8,$g_1$&90.5&86.0&97.5&97.5&17.4&17.4&18.5&18.5 \\ \hline
7&1000,5,8,$g_1$&100&95.1&100&100&3.8&4.6&3.4&3.4 \\ \hline
8&1000,15,8,$g_1$&50.7&47.5&61.8&61.4&4.5&4.5&4.5&4.5 \\ \hline
9&250,5,32,$g_1$&100&100&100&100&4.0&6.5&4.8&4.7 \\ \hline
10&250,15,32,$g_1$&92.6&89.9&98.8&98.7&15.1&16.4&13.7&13.7 \\ \hline
11&1000,5,32,$g_1$&99.9&97.6&100&100&1.8&3.0&1.9&1.9 \\ \hline
12&1000,15,32,$g_1$&54.0&49.7&65.6&65.6&4.7&4.6&4.7&4.8 \\ \hline
13&250,5,2,$g_2$&98.0&85.0&99.2&83.0&11.1&11.2&9.8&8.7 \\ \hline
14&250,15,2,$g_2$&55.2&51.8&65.2&53.7&13.0&12.9&12.6&11.2 \\ \hline
15&1000,5,2,$g_2$&88.4&68.4&94.1&65.4&3.6&3.4&3.3&2.7 \\ \hline
16&1000,15,2,$g_2$&29.8&27.8&36.1&29.1&3.5&3.6&3.3&3.0 \\ \hline
17&250,5,8,$g_2$&100&99.3&100&91&8.2&10.4&9.8&9.5 \\ \hline
18&250,15,8,$g_2$&89.7&85.1&97.1&74.2&16.8&16.9&18.5&14.3 \\ \hline
19&1000,5,8,$g_2$&99.9&94.4&100&77.6&2.9&3.5&3.5&3.0 \\ \hline
20&1000,15,8,$g_2$&49.3&47.0&62.5&43.2&4.2&4.3&4.6&3.8 \\ \hline
21&250,5,32,$g_2$&100&99.8&100&91.1&1.4&3.3&4.7&9.4 \\ \hline
22&250,15,32,$g_3$&93.2&89.5&98.9&76.4&14.4&15.4&13.6&14.5 \\ \hline
23&1000,5,32,$g_2$&99.9&96.6&100&78.8&0.8&2.0&1.9&3.0 \\ \hline
24&1000,15,32,$g_2$&51.8&48.5&66.4&44.6&4.3&4.3&4.7&3.8 \\ \hline
25&250,5,2,$g_3$&97.9&85.3&99.2&68.7&11.5&11.6&9.9&8.0 \\ \hline
26&250,15,2,$g_3$&57.8&53.2&65.6&46.6&13.6&13.1&12.7&10.5 \\ \hline
27&1000,5,2,$g_3$&89.3&67.9&93.8&47.7&3.9&3.7&3.3&2.4 \\ \hline
28&1000,15,2,$g_3$&30.8&28.7&35.7&24.1&3.7&3.7&3.4&2.8 \\ \hline
29&250,5,8,$g_3$&99.9&99.2&100&76.2&0.5&13.0&9.5&8.5 \\ \hline
30&250,15,8,$g_3$&90.6&85.4&97.7&59.8&17.6&17.8&18.7&12.3 \\ \hline
31&1000,5,8,$g_3$&100&95.1&99.9&56.0&3.7&4.4&3.4&2.6 \\ \hline
32&1000,15,8,$g_3$&51.1&47.2&62.5&32.2&4.6&4.5&4.6&3.3 \\ \hline
33&250,5,32,$g_3$&100&99.9&100&75.9&2.7&5.3&4.7&8.6 \\ \hline
34&250,15,32,$g_3$&93.1&89.6&98.8&60.8&15.2&16.4&13.7&12.5 \\ \hline
35&1000,5,32,$g_3$&99.9&97.5&100&57.5&1.4&2.8&1.9&2.6 \\ \hline
36&1000,15,32,$g_3$&54.0&49.8&66.1&33.4&4.7&4.6&4.8&3.3 \\ \hline
\hline
\end{tabular}
\caption{For the pseudo real data from dataset Leukemia performance of Lasso with (LK) and without (LN) prior 
knowledge of the nuisance function are compared with our estimator for both independent (DPi) and dependent 
(DPd) settings. True and false positive rates (TPR and FPR respectively) are given in this table. The results 
are based on 1000 replicates for each design.}
\label{ResultsLe2.table}
\end{table}

\begin{table} \footnotesize \renewcommand{\tabcolsep}{3.5pt} \centering 
\begin{tabular}{||r|l || c|c|c|c || c|c|c|c || c|c ||llll}
\hline\hline
\multicolumn{12}{||c||}{Results for the set Prostate}\\
\hline\hline
$n^\circ$&setting & \multicolumn{4}{c||}{Estimation error}&\multicolumn{4}{c||}{Prediction
error}&\multicolumn{2}{c||}{$\|\hat g-g^0\|_n$}\\
\hline
$N$&$p$,$s_0$,STN,$g_j$&DPi&DPd&LK&LN &DPi&DPd&LK&LN &DPi&DPd\\
\hline \hline
1&250,5,2,$g_1$&0.70&0.74&0.59&0.60&4.66&5.41&4.16&4.25&0.46&0.82 \\ \hline
2&250,15,2,$g_1$&1.47&1.51&1.36&1.36&16.82&17.13&16.01&16.01&0.92&1.25 \\ \hline
3&1000,5,2,$g_1$&0.79&0.82&0.69&0.70&6.11&6.53&5.37&5.41&0.49&0.96 \\ \hline
4&1000,15,2,$g_1$&1.61&1.64&1.53&1.53&19.72&20.06&18.93&18.93&0.98&1.33 \\ \hline
5&250,5,8,$g_1$&0.18&0.19&0.15&0.15&1.23&1.78&1.07&1.09&0.12&0.31 \\ \hline
6&250,15,8,$g_1$&0.42&0.44&0.39&0.39&8.25&9.16&6.99&7.09&0.31&0.64 \\ \hline
7&1000,5,8,$g_1$&0.21&0.23&0.18&0.18&20.7&3.15&1.75&1.78&0.14&0.50 \\ \hline
8&1000,15,8,$g_1$&0.54&0.55&0.51&0.51&15.59&15.86&13.91&14.02&0.49&0.93 \\ \hline
9&250,5,32,$g_1$&0.06&0.06&0.05&0.05&0.33&0.57&0.28&0.29&0.03&0.15 \\ \hline
10&250,15,32,$g_1$&0.21&0.21&0.20&0.20&4.70&5.46&3.47&3.54&0.16&0.46 \\ \hline
11&1000,5,32,$g_1$&0.06&0.08&0.06&0.06&0.61&1.25&0.46&0.47&0.04&0.26 \\ \hline
12&1000,15,32,$g_1$&0.29&0.30&0.28&0.28&13.73&14.43&11.27&11.42&0.40&0.84 \\ \hline
13&250,5,2,$g_2$&0.71&0.74&0.59&11.74&4.65&5.37&4.20&6.16&0.47&0.84 \\ \hline
14&250,15,2,$g_2$&1.49&1.50&1.35&11.80&16.90&17.15&15.88&17.19&0.93&1.25 \\ \hline
15&1000,5,2,$g_2$&0.80&0.81&0.70&11.75&6.06&6.58&5.48&7.07&0.49&0.93 \\ \hline
16&1000,15,2,$g_2$&1.62&1.63&1.53&11.82&19.71&19.90&18.94&19.60&0.99&1.33 \\ \hline
17&250,5,8,$g_2$&0.18&0.20&0.15&11.73&1.19&1.81&1.07&5.51&0.12&0.32 \\ \hline
18&250,15,8,$g_2$&0.42&0.43&0.38&11.72&8.27&9.33&6.92&14.92&0.31&0.66 \\ \hline
19&1000,5,8,$g_2$&0.21&0.23&0.18&11.72&1.97&3.09&1.73&6.51&0.14&0.50 \\ \hline
20&1000,15,8,$g_2$&0.54&0.56&0.51&11.73&15.62&15.96&13.92&18.39&0.50&0.95 \\ \hline
21&250,5,32,$g_2$&0.15&0.15&0.05&11.72&0.69&1.11&0.27&5.38&0.07&0.29 \\ \hline
22&250,15,32,$g_3$&0.22&0.23&0.20&11.71&5.1&5.93&3.52&14.58&0.18&0.49 \\ \hline
23&1000,5,32,$g_2$&0.15&0.16&0.06&11.72&1.07&1.89&0.44&6.48&0.08&0.41 \\ \hline
24&1000,15,32,$g_2$&0.30&0.31&0.28&11.72&13.86&14.55&11.23&18.22&0.41&0.84 \\ \hline
25&250,5,2,$g_3$&0.71&0.74&0.59&3.99&4.72&5.47&4.20&6.69&0.47&0.85 \\ \hline
26&250,15,2,$g_3$&1.49&1.49&1.36&4.13&17.08&17.06&15.83&18.05&0.93&1.25 \\ \hline
27&1000,5,2,$g_3$&0.81&0.81&0.70&3.97&6.06&6.54&5.37&7.81&0.50&0.94 \\ \hline
28&1000,15,2,$g_3$&1.63&1.62&1.53&4.2&19.86&19.93&19.00&20.05&1.00&1.33 \\ \hline
29&250,5,8,$g_3$&0.18&0.19&0.15&3.94&1.24&1.86&1.09&6.38&0.12&0.32 \\ \hline
30&250,15,8,$g_3$&0.42&0.44&0.39&3.91&8.22&9.31&6.98&16.78&0.31&0.66 \\ \hline
31&1000,5,8,$g_3$&0.21&0.23&0.18&3.91&2.09&3.11&1.75&7.46&0.14&0.49 \\ \hline
32&1000,15,8,$g_3$&0.53&0.54&0.49&3.96&15.36&15.98&13.82&19.37&0.49&0.91 \\ \hline
33&250,5,32,$g_3$&0.08&0.09&0.05&3.94&0.42&0.74&0.28&6.31&0.04&0.19 \\ \hline
34&250,15,32,$g_3$&0.20&0.21&0.19&3.89&4.68&5.56&3.39&16.64&0.16&0.48 \\ \hline
35&1000,5,32,$g_3$&0.09&0.10&0.06&3.91&0.74&1.45&0.48&7.47&0.05&0.30 \\ \hline
36&1000,15,32,$g_3$&0.28&0.30&0.27&3.92&14.05&14.59&11.38&19.41&0.41&0.87 \\ \hline
\hline
\end{tabular}
\caption{The results of the pseudo real data study with the dataset Prostate are presented here. For 
the 36 different designs ($N=1,\ldots,36$) the performance of Lasso with (LK) and without (LN) prior 
knowledge on the nuisance function are compared with our estimator in both the independent (DPi) and dependent 
(DPd) case. Prediction and estimation error for $\hat\beta$ and $\hat g$ are given. The results are based 
on 1000 replicates for each design.}
\label{ResultsPr1.table}
\end{table}

\begin{table} \footnotesize \centering 
\begin{tabular}{||r|l || c|c|c|c || c|c|c|c || }
\hline\hline
\multicolumn{10}{||c||}{Results for the set Prostate}\\
\hline\hline
$n^\circ$&setting &\multicolumn{4}{c||}{TPR}&\multicolumn{4}{c||}{FPR}\\
\hline
$N$&$p$,$s_0$,STN,$g_i$&DPi&DPd&LK&LN &DPi&DPd&LK&LN\\
\hline \hline
1&250,5,2,$g_1$&89.8&79.3&93.0&92.7&10.8&11.2&10.3&10.4 \\ \hline
2&250,15,2,$g_1$&61.8&56.9&67.2&66.1&13.6&13.3&14.0&13.8 \\ \hline
3&1000,5,2,$g_1$&76.1&61.0&82.3&81.2&3.7&3.5&3.5&3.4 \\ \hline
4&1000,15,2,$g_1$&39.0&34.0&43.4&42.7&4.1&4.1&4.1&4.0 \\ \hline
5&250,5,8,$g_1$&99.9&98.3&99.9&99.9&11.3&13.8&10.5&10.6 \\ \hline
6&250,15,8,$g_1$&94.2&91.3&96.5&96.5&18.9&19.4&19.1&19.1 \\ \hline
7&1000,5,8,$g_1$&98.4&91.9&99.2&99.2&4.4&5.2&4.2&4.2 \\ \hline
8&1000,15,8,$g_1$&68.5&64.5&76.9&76.3&6.1&6.1&6.5&6.4 \\ \hline
9&250,5,32,$g_1$&100&99.1&100&100&4.9&7.1&4.5&4.5 \\ \hline
10&250,15,32,$g_1$&97.8&96.4&99.0&99.0&13.7&15.1&12.3&12.4 \\ \hline
11&1000,5,32,$g_1$&99.5&97.1&99.8&99.8&2.4&3.5&2.2&2.2 \\ \hline
12&1000,15,32,$g_1$&75.5&70.9&83.9&83.3&6.2&6.3&6.4&6.4 \\ \hline
13&250,5,2,$g_2$&89.9&77.7&92.5&73.5&10.7&11.1&10.5&8.7 \\ \hline
14&250,15,2,$g_2$&61.6&56.3&66.6&55.1&13.7&13.4&13.8&11.8 \\ \hline
15&1000,5,2,$g_2$&77.3&60.5&82.8&58.1&3.7&3.5&3.6&2.8 \\ \hline
16&1000,15,2,$g_2$&38.4&35.1&43.6&34.2&4.1&4.1&4.1&3.5 \\ \hline
17&250,5,8,$g_2$&99.9&98.2&99.9&80.8&10.8&13.4&10.6&9.5 \\ \hline
18&250,15,8,$g_2$&94.3&91.3&96.9&72.8&19.4&20.0&19.2&15.4 \\ \hline
19&1000,5,8,$g_2$&98.5&92.7&99.3&67.6&3.9&4.8&4.2&3.1 \\ \hline
20&1000,15,8,$g_2$&69.0&64.0&76.7&50.1&6.2&6.1&6.5&4.6 \\ \hline
21&250,5,32,$g_2$&99.9&98.7&100&81.7&2.7&4.9&4.5&9.4 \\ \hline
22&250,15,32,$g_3$&97.7&96.2&98.9&74.5&13.8&15.3&12.4&15.5 \\ \hline
23&1000,5,32,$g_2$&99.5&95.6&99.9&68.9&1.5&2.6&2.2&3.2 \\ \hline
24&1000,15,32,$g_2$&75.6&71.2&83.8&52.3&6.3&6.4&6.4&4.8 \\ \hline
25&250,5,2,$g_3$&90.5&78.3&93.2&62.4&11.0&11.1&10.5&7.6 \\ \hline
26&250,15,2,$g_3$&61.3&56.2&66.3&47.7&13.9&13.1&13.6&10.8 \\ \hline
27&1000,5,2,$g_3$&76.8&60.8&82.6&45.9&3.6&3.6&3.5&2.5 \\ \hline
28&1000,15,2,$g_3$&38.3&34.0&43.6&27.8&4.1&4.1&4.1&3.2 \\ \hline
29&250,5,8,$g_3$&99.8&98.3&99.9&67.1&11.4&14.3&10.9&8.1 \\ \hline
30&250,15,8,$g_3$&94.1&91.3&96.5&59.8&19.2&19.7&19.1&12.9 \\ \hline
31&1000,5,8,$g_3$&98.1&92.9&99.1&52.2&4.4&5.2&4.1&2.7 \\ \hline
32&1000,15,8,$g_3$&70.2&64.6&77.6&37.7&6.2&6.2&6.6&3.8 \\ \hline
33&250,5,32,$g_3$&99.9&99.0&100&68.2&3.6&6.2&4.5&8.0 \\ \hline
34&250,15,32,$g_3$&97.8&96.4&99.1&61.2&14.0&15.3&12.4&13.2 \\ \hline
35&1000,5,32,$g_3$&99.5&96.8&99.8&52.6&1.9&3.2&2.2&2.7 \\ \hline
36&1000,15,32,$g_3$&75.6&70.6&83.9&39.1&6.4&6.4&6.4&3.9 \\ \hline
\hline
\end{tabular}
\caption{For the pseudo real data from dataset Prostate performance of Lasso with (LK) and without (LN) prior 
knowledge of the nuisance function are compared with our estimator for both independent (DPi) and dependent 
(DPd) settings. True and false positive rates (TPR and FPR respectively) are given in this table. The results 
are based on 1000 replicates for each design.}
\label{ResultsPr2.table}
\end{table}

\clearpage

\bibliographystyle{apalike} 
\bibliography{reference}

\vspace{0.8cm}
\noindent Patric M\"uller, Seminar f\"ur Statistik, Department of Mathematics, ETH Zurich, CH-8092 Zurich, 
Switzerland.\\
E-mail: muellepa@stat.math.ethz.ch

\end{document}